\documentclass{elsarticle}
\usepackage{geometry}                % See geometry.pdf to learn the layout options. There are lots.
\usepackage{graphicx}
\usepackage[reqno]{amsmath}% AMS-LaTeX-first 
\usepackage{amssymb}
\usepackage{epstopdf}
\usepackage{amsthm}
\usepackage[]{graphicx}               % graphics
\usepackage{bm}
\usepackage{esvect}
\usepackage{ulem}
\usepackage{adjustbox}
\usepackage{accents}
\usepackage{stmaryrd}
\usepackage{xcolor}
\usepackage{enumitem}
\usepackage{hyperref}
\usepackage{booktabs}
\usepackage{enumitem}
\definecolor{chcol}{rgb}{0.4,0.,0.9}

\newcommand{\halfComma}{\kern 0.083334em}
\newcommand\iprod[1]{\left\langle #1\right\rangle}                                             % inner product
\newcommand\inorm[1]{\left |\left| #1\right|\right|}                                               % norm
\newcommand\spacevec[1]{\accentset{\,\rightarrow}{#1}}                        % spatial vector, i.e. x \hat x + y \hat y + z \hat z
\newcommand\statevec[1]{\mathbf #1}                                                     % state vector, e.g. [rho, rho \vec{v}, E]^T
\newcommand\statevecGreek[1]{\boldsymbol #1}                                     % state vector for Greek symbols
\newcommand\mmatrix[1]{\underbar{#1}}				% Matrix as taught in linear algebra, i.e. math matrix.
\newcommand\oneHalf{\frac{1}{2}}

\usepackage{xargs}    
\newtheorem{thm}{Theorem}
\newtheorem{lem}{Lemma}
\newtheorem{rem}{Remark}
\newtheorem{cor}{Corollary}
\newtheorem{ex}{Example}

\begin{document}

\begin{frontmatter}

\title{On the Theoretical Foundation of Overset Grid Methods for Hyperbolic Problems: Well-Posedness and Conservation}
%\titlerunning{Theoretical Foundation of Overset Grid Methods}

\author[1]{David A. Kopriva}
\ead{kopriva@math.fsu.edu}
\address[1]{Department of Mathematics, The Florida State University, Tallahassee, FL 32306, USA and Computational Science Research Center, San Diego State University, San Diego, CA, USA.}

\author[2]{Jan  Nordstr\"om\corref{cor2}}
\ead{jan.nordstrom@liu.se}
\address[2]{Department of Mathematics, Applied Mathematics, Linköping University, 581 83 Linköping, Sweden and Department of Mathematics and Applied Mathematics
 University of Johannesburg
 P.O. Box 524, Auckland Park 2006, South Africa.}
\author[3]{Gregor J. Gassner}
\ead{ggassner@uni-koeln.de}
\address[3]{Department for Mathematics and Computer Science; Center for Data and Simulation Science,
 University of Cologne, Weyertal 86-90, 50931, Cologne, Germany.}
 \cortext[cor2]{Corresponding author}
\begin{abstract}
We use the energy method to study the well-posedness of initial-boundary value problems approximated by overset mesh methods in one and two space dimensions for linear constant-coefficient hyperbolic systems. We show that in one space dimension, for both scalar equations and systems of equations, the problem where one domain partially oversets another is well-posed when characteristic coupling conditions are used. If a system cannot be diagonalized, as is ususally the case in multiple space dimensions, then the energy method does not give proper bounds in terms of initial and boundary data. For those problems, we propose a novel penalty approach. 
We show, by using a global energy that accounts for the energy in the overlap region of the domains, that under well-defined conditions on the coupling matrices the penalized overset domain problems are energy bounded, conservative, well-posed and have solutions equivalent to the original single domain problem.

\end{abstract}
\end{frontmatter}

%\author{David A. Kopriva, Jan Nordstr\"om, Gregor J. Gassner}
% \institute{David A. Kopriva \at Department of Mathematics, The Florida State University, Tallahassee, FL 32306, USA and Computational Science Research Center, San Diego State University, San Diego, CA, USA. \email{kopriva@math.fsu.edu}\\
% Jan  Nordstr\"om \at Department of Mathematics, Applied Mathematics, Linköping University, 581 83 Linköping, Sweden and Department of Mathematics and Applied Mathematics
% University of Johannesburg
% P.O. Box 524, Auckland Park 2006, South Africa. \email{jan.nordstrom@liu.se}\\
%  \mailname{jan.nordstrom@liu.se}\\Gregor J. Gassner \at Department for Mathematics and Computer Science; Center for Data and Simulation Science,
% University of Cologne, Weyertal 86-90, 50931, Cologne, Germany. \email{ggassner@uni-koeln.de}
%           }
%\date{}                                           % Activate to display a given date or no date

%\maketitle
\section{Introduction}

Overset grid methods \cite{Chan99},\cite{Meakin:1999io} have been used for forty years to simplify the application of numerical methods to complex geometric configurations. A typical overset grid model is to use a fitted grid near a body, placed over a simpler underlying mesh, with a region of the underlying mesh blanked out. An example of such a mesh is shown in Fig. \ref{fig:oversetSketch}. Interface conditions are applied to the overset mesh boundaries to couple the solutions between them. 

\begin{figure}[tbp] %  figure placement: here, top, bottom, or page
   \centering
   \includegraphics[width=3.5in]{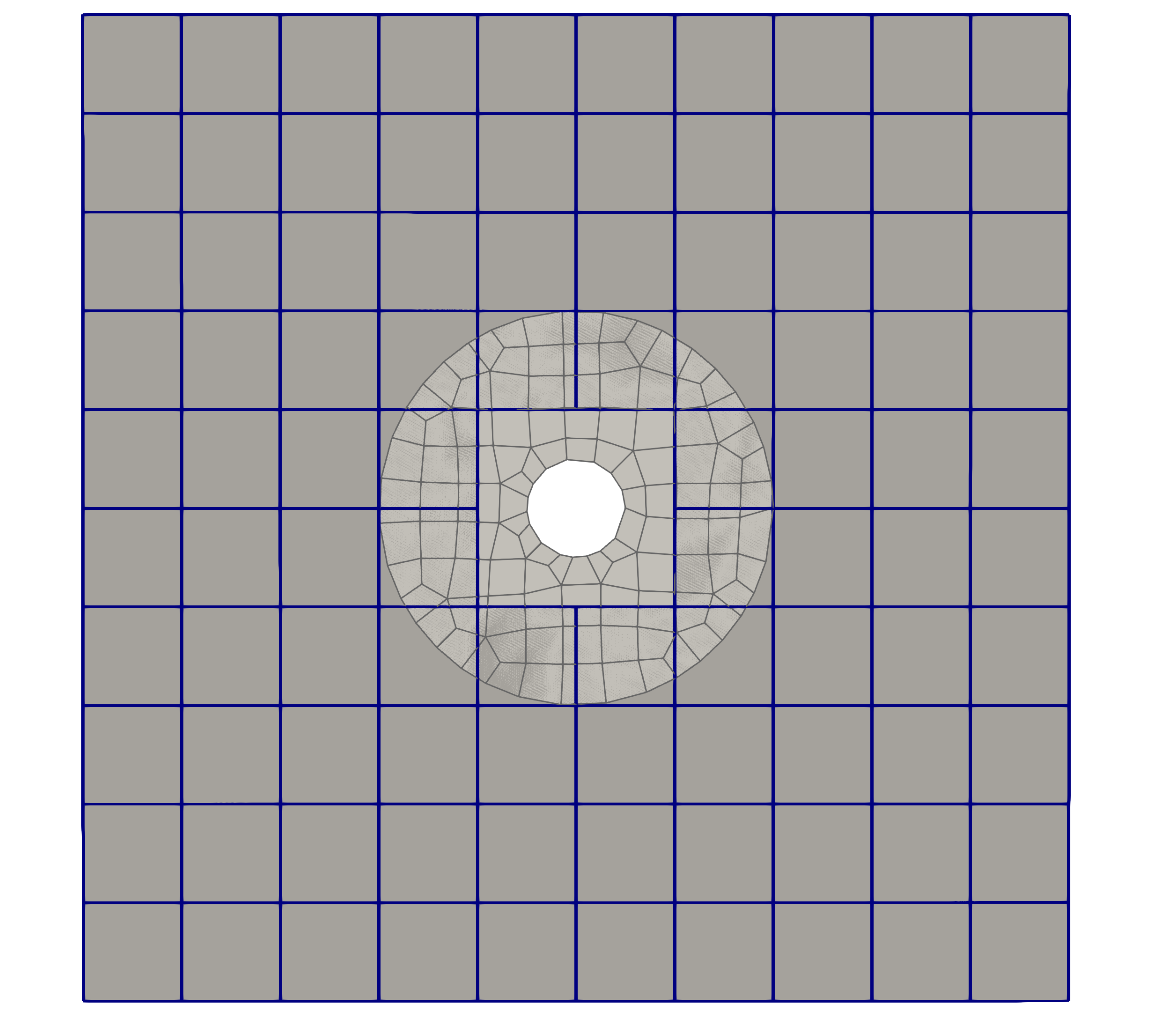} 
   \caption{Example of an overset grid where a mesh with two circular boundaries oversets a uniform Cartesian grid with a blanked out region in the center}
   \label{fig:oversetSketch}
\end{figure}

The history and literature of overset grid methods are extensive. The topic has its own conference series \cite{OSETGrid} that has been held for almost thirty years. Numerous software packages exist to implement the schemes, including \cite{ChimeraGridtools}, \cite{Buning:2004ef}, \cite{Henshaw:2002lk}, to cite only three. The methods have been used in conjunction with all major spatial approximation schemes: finite difference \cite{STEGER1987301}, finite element (continuous \cite{Hansbo:2003eu}, discontinuous \cite{GALBRAITH201427},\cite{brazell2016overset}, and hybridizable \cite{Kauffman:2017uj}), finite volume \cite{OSTI}, and spectral \cite{Kopriva:1989},\cite{Merrill:2016qj}. Finally, overset grid methods are used in a wide variety of application areas including aerodynamics \cite{STEGER1987301}, solid mechanics \cite{Appelo:2012du}, meteorology \cite{kageyama2004yin}, and electromagnetics \cite{Henshaw:2006xt}.

Stability of the coupling procedures has long been a practical and theoretical issue with overset mesh methods \cite{SHARAN2018199}, and to date fully multidimensional stability proofs are not available. The most recent works have used the energy method with summation by parts finite difference schemes to prove stability for one-dimensional scalar and system problems \cite{Bodony:2011cl}. An extension to two space dimensions is made in \cite{Reichert:2011ru}, with a note that the one dimensional proof extends only to problems where the coefficient matrices are simultaneously diagonalizable. 

One of the issues with finding a general stability proof is that the original initial-boundary value problem (IBVP) on which the approximation is based needs to be well-posed in the first place. If it is not, then further development is pointless. In fact, the steps used to show well-posedness can often be directly followed to lead to a stability proof for a numerical scheme \cite{Nordstrom:2016jk}. We are unaware, however, of any work to date that has examined the well-posedness of overset grid problems on which the numerical methods are based. 

In this paper, we step back from the numerical issue of stability and examine the more fundamental problem on the overset mesh approach as an IBVP problem of overset subdomains, and formulate such problems that are energy bounded, well-posed, and conservative.
% A well-formulated PDE problem is necessary for convergence and can be used to provide guidence on how to construct stable numerical approximations.

We broadly describe the overset domain problem that we address in the following way: The solution, $\statevecGreek\omega$, to the hyperbolic system of equations
\begin{equation}
    \statevecGreek \omega_t + \sum_{i-1}^d \mmatrix A_i\frac{\partial \statevecGreek \omega}{\partial x_i} = 0
\end{equation}
is desired on a domain $\Omega$, where $d$ is the number of space dimensions, $\{ x_i\}_{i=1}^d$ are the coordinate directions, and the coefficient matrices, $\mmatrix A_i$, are constant and symmetric. We \emph{do not} require that the coefficient matrices are simultaneously diagonalizable. 

Instead of solving the problem on $\Omega$, it is solved on a series of problems posed on a collection of subdomains, $\Omega_k$, that can overlap, but which cover $\Omega$, i.e. $\Omega = \cup \Omega_k$. Each subdomain hosts a solution $\statevec u^k$ that satisfies 
the same linear system of hyperbolic partial differential equations,
\begin{equation}
    \statevec u^k_t + \sum_{i=1}^d \mmatrix A_i\frac{\partial\statevec u^k}{\partial x_i} = 0.
\end{equation}
Boundary conditions are applied along the physical boundaries $\partial \Omega \displaystyle \cap \Omega_k$, which we assume are dissipative. 

Importantly, initial conditions are specified over the whole domain, $\Omega$, so that in any overlap region the solutions within the subdomains coincide at time $t = 0$ with the global solution, $\statevecGreek\omega_0 = \statevecGreek\omega(\spacevec x,0)$. Coupling conditions are applied to the subdomain intersections, and possibly within the overlap regions themselves. The general expectation among practitioners when overset grids are used is that the solutions evolve within each subdomain equivalently so that they represent the solution one would get by solving the problem on the original domain, $\Omega$. We investigate that assumption in detail and pose problems for which it is valid.

More specifically, we restrict the problems in this paper to two subdomains with solutions $\statevec u$ and $\statevec v$ in one and two space dimensions that partially overlap, but which otherwise contain the salient features of the general problem. Throughout the paper we assume trivial upwind physical boundary conditions, and focus solely on the subdomain coupling.

We develop the results systematically, starting with the problem of two domains with partial overlap in one space dimension for the scalar equation, which provides insight, and allows us to introduce the approach and technique that we use. For one space dimension, we show that the problem where characteristic interface conditions are used to transfer data from one subdomain to another is energy bounded for scalar problems and for the hyperbolic system. The system proof relies on the diagonalizability of the coefficient matrix. We proceed to show that if one does not use diagonalizability, which one cannot in general do in multiple space dimensions, then parasitic terms appear that are not bounded by initial or boundary data.

To get energy bounded problems without assuming simultaneous diagonalization of the coefficient matrices, we use multiple interior penalty functions to eliminate terms not bounded by the data. Multiple penalty procedures have been used previously to increase the accuracy and stability of numerical solutions \cite{FRENANDER201634,Nordstrom:2014qd}, and here we use them as part of the original IBVP to enforce energy boundedness. By adding penalty terms to the original equations we can remove the parasitic terms and obtain a suitable energy bound. 

The use of the penalty terms extends to multiple space dimensions (without the assumption of diagonalization), and we show well-posedness of two-dimensional overset domain problems with penalties applied at the overlap interfaces and in the interior of the overlap region. We present two multiple penalty formulations of the overset domain problem: The first applies penalties only at the boundaries of the overlap region. The other introduces additional penalties within the overlap region. The additional penalties could be viewed as favoring one subdomain over another, e.g. to favor one, numerically, that is expected to be more accurate due to better resolution. We show that with specific conditions on the coupling matrices and for trivial upwind physical boundary conditions, both formulations have their energy bounded by the initial data. Furthermore, the coupling conditions required for energy boundedness simultaneously ensure conservation, so that the flux lost from one subdomain is gained by the other.

We show well-posedness and conservation by first developing an energy bound for the combined overset domain solutions. Unique in our approach is to use three techniques for the general problem: We
\begin{enumerate}
    \item[(T1)] account for the energy in the overset portion and relate it to interface values,
    \item[(T2)] use a norm that corrects for the double--counted energy due to the overlap to allow bounds to be made on the individual domain components, and
    \item[(T3)] use internal penalties to get desired energy bounds. Although the use of penalty terms is common, we use them here for the first time on the overset domain problem.
\end{enumerate}

 We use the energy estimates to show that the solutions of the overset problems match the solution of the original domain problem. From the fact that the original problem is well-posed, i.e. the solution exists, is unique, and is energy bounded, we infer that the overset domain problem is also well-posed. Finally, we show that in each case, the overset domain problems with penalty conditions are conservative, if they are stable.

To aid with navigating the paper, we collect the overset domain problems that we study in Table \ref{Tab:Summary}, and list the theorems for well-posedness and conservation. 
\begin{table}[htp]
\begin{center}\begin{tabular}{l|llll}
 %\hline
Equation & Coupling & Dimension & Well-Posedness & Conservation \\
 \hline
Scalar & Characteristic & 1D & Thm. \ref{thm:WellPosednessScalar1D} &  \\System & Characteristic & 1D & Thm. \ref{thm:1DSystemWellPosedness} &  \\
System & Interface Penalty & 1D & Thm. \ref{thm:1DEquivalenceWithPenalty} & Thm. \ref{thm:1DPenaltyConservationMatch} \\System & Interface Penalty & 2D & Thm. \ref{thm:2DEquivalence} & Thm. \ref{thm:Conservation2DOriginal} \\
System & Interface + Overlap Penalty & 2D & Thm. \ref{thm:Wellposed2DWithOverlapPenalty} & Thm. \ref{thm:2DOverlapPenaltyConservation}
\end{tabular} 
\end{center}
\caption{Collection of results for well-posed overset domain problems}
\label{Tab:Summary}
\end{table}

\section{Overset Domains in One Space Dimension}
We start by examining three one-dimensional overset domain problems. The first is for the scalar equation. The second extends the problem to the system, and shows energy boundedness using the fact that the system can be diagonalized. That problem is re-done without the use of diagonalization, which is usually not available in multiple space dimensions, to show that additional terms appear in the energy that are not bounded by initial or boundary data. Finally, we introduce a multiple penalty formulation for the system that is energy bounded without the need to use diagonalization, making it a candidate for multidimensional problems.

\subsection{The Scalar Problem in One Space Dimension}

The scalar overset domain problem in one space dimension seeks the solution to the initial-boundary-value problem
\begin{equation}
\begin{gathered}
    \omega_t +\alpha \omega_x = 0,\quad x\in\Omega = (a,d)\hfill\\
    \omega(a,t)=0,\quad t>0\hfill\\
    \omega(x,0)= \omega_0, \quad x\in\Omega,\hfill
\end{gathered}
\label{eq:ScalarOmegaBVP}
\end{equation}
where $\alpha>0$, by finding the solutions on two domains $\Omega_u$ and $\Omega_v$, as seen in Fig. \ref{fig:Overlap1D}.
\begin{figure}[tbp] %  figure placement: here, top, bottom, or page
   \centering
   \includegraphics[width=4in]{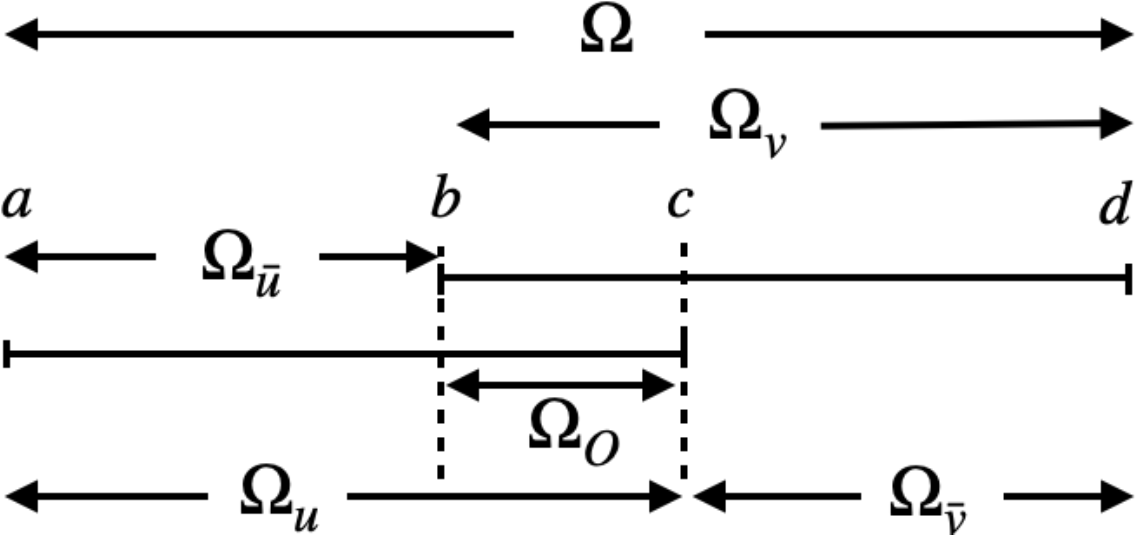} 
   \caption{Overset domain definitions in 1D}
   \label{fig:Overlap1D}
\end{figure}
For those domains, we pose two problems, $L$ and $R$
\begin{equation}
L\;\left\{\begin{gathered}
u_t + \alpha u_x = 0, \quad x\in\Omega_u\hfill\\
u(a,t) = 0\hfill\\
u(x,0) = \omega_0(x) \hfill
\end{gathered} \right. \quad
R\;\left\{\begin{gathered}
v_t + \alpha v_x = 0, \quad x\in\Omega_v\hfill\\
v(b,t) = u(b,t)\hfill\\
v(x,0) = \omega_0(x) \hfill
\end{gathered} \right.
\label{eq:L-RSystem}
\end{equation}

The original problem \eqref{eq:ScalarOmegaBVP} for $\omega$ is known to be well-posed and the energy is bounded by the initial energy \cite{Lorenz:1989fk}. We show that $u$ and $v$ are energy bounded, and from that, $u(x,T) = \omega(x,T)$ on $\Omega_u$ and $v(x,T) = \omega(x,T)$ on $\Omega_v$, making the overset problem well-posed as well.

We first show that the problems posed in \eqref{eq:L-RSystem} are energy bounded. Showing boundedness of the problem $L$ is standard: We multiply the PDE by $u$ and integrate over the domain $\Omega_u$ to get
\begin{equation}
    \iprod{u,u_t}_{\Omega_u} + \alpha\iprod{u,u_x}_{\Omega_u} = 0,
\end{equation}
where $\iprod{f,g}_{\Omega_u} = \int_{\Omega_u} fgdx$. The energy is then defined as the norm, $\inorm{u}_{\Omega_u} = \iprod{u,u}^{1/2}_{\Omega_u}$.
Integrating by parts and applying the boundary condition on the left,
\begin{equation}
\frac{d}{dt}\inorm{u}_{\Omega_u}^2 = -\alpha u^2(c,t).
\label{eq:ScalarUNormDt}
\end{equation}
When we integrate \eqref{eq:ScalarUNormDt} in time, 
\begin{equation}
    \inorm{u(T)}^2_{\Omega_u}+\alpha\int_0^T u^2(c,t)dt \le \inorm{\omega_0}^2_{\Omega_u},
\end{equation}
which shows that the energy at any time $T$ and the value $u(c,T)$ are bounded by the initial condition. The problem for $u$ is completely decoupled from the problem for $v$.

The problem, $R$, has a non-trivial inflow boundary condition, so the same procedure leads to
\begin{equation}
\frac{d}{dt}\inorm{v}_{\Omega_v}^2 = -\alpha v^2(d,t) + \alpha u^2(b,t),
\label{eq:VEnergyWithUBC}
\end{equation}
where we have applied the boundary condition, $v(b,t) = u(b,t)$.
To show energy boundedness, i.e. the solution is bounded in terms of the data, we need a bound on $u(b,t)$.

As the key step to get a bound on the interior value, $u(b,t)$, we form the energy on the complementary domain $\Omega_{\bar u}$, and use integration by parts again, noting that integration by parts applies on any subinterval of the domain $\Omega_u$. 
Thus, it is also true, as in \eqref{eq:ScalarUNormDt}, that
\begin{equation}
\frac{d}{dt}\inorm{u}_{\Omega_{\bar u}}^2 = -\alpha u^2(b,t).
\label{eq:PartialBoundL}
\end{equation}
The energy rate \eqref{eq:PartialBoundL} expresses the use of technique (T1).
Therefore, we can write the value at $x=b$ in terms of the time derivative of the energy in the complementary domain, $\Omega_{\bar u}$ as
\begin{equation}
\alpha u^2(b,t) = -\frac{d}{dt}\inorm{u}_{\Omega_{\bar u}}^2.
\label{eq:ubarvalue}
\end{equation}
%This is technique (T1) noted in the introduction.

When we insert \eqref{eq:ubarvalue} into \eqref{eq:VEnergyWithUBC}, and acknowledge the dissipativity of the (outflow) term on the right at $x=d$, we bound $v$ in terms of $u$, as
\begin{equation}
\frac{d}{dt}\inorm{v}_{\Omega_v}^2 \le  -\frac{d}{dt}\inorm{u}_{\Omega_{\bar u}}^2,
\end{equation}
or
\begin{equation}
\frac{d}{dt}\inorm{v}_{\Omega_v}^2 + \frac{d}{dt}\inorm{u}_{\Omega_{\bar u}}^2\le  0.
\label{eq:scalarVboundeqn}
\end{equation}
Alternately, we can simply add \eqref{eq:VEnergyWithUBC} and \eqref{eq:PartialBoundL} to get \eqref{eq:scalarVboundeqn}.

We then integrate \eqref{eq:scalarVboundeqn} in time, giving
\begin{equation}
\inorm{v}_{\Omega_v}^2 + \inorm{u}_{\Omega_{\bar u}}^2\le \inorm{\omega_0}_{\Omega_v}^2 + \inorm{\omega_0}_{\Omega_{\bar u}}^2 .
\label{eq:Sharp1DScalarEstimateV}
\end{equation}
Now, $\inorm{\omega_0}_{\Omega_{\bar u}}^2\le  \inorm{\omega_0}_{\Omega_u}^2$, so
\begin{equation}
\inorm{v}_{\Omega_v}^2 \le \inorm{\omega_0}_{\Omega_v}^2 + \inorm{\omega_0}_{\Omega_u}^2.
\label{eq:NotSharp1DScalarEstimate}
\end{equation}
Therefore, each of the problems $L$ and $R$ are energy bounded, with the norm of the solution being bounded by the initial data (for trivial inflow conditions). 

\begin{rem}
The right hand side of \eqref{eq:NotSharp1DScalarEstimate} is not sharp when using the whole domain since it double counts the overlap region $\Omega_O$. It's also not necessary. The sharper bound is to use \eqref{eq:Sharp1DScalarEstimateV}, which does not include the double counted energy,
\begin{equation}
\begin{split}
& \inorm{u}_{\Omega_u}^2 \le \inorm{\omega_0}_{\Omega_u}^2
\\&
\inorm{v}_{\Omega_v}^2 \le \inorm{\omega_0}_{\Omega_v}^2 +  \inorm{\omega_0}_{\Omega_{\bar u}}^2 = \inorm{\omega_0}^2_{\Omega}.
\end{split}
\label{eq:Sharp1DScalarEstimate}
\end{equation}
This is technique (T2).
The sharper result makes sense: The total energy in $\Omega_u$ is bounded by the initial value. The total energy in $\Omega_v$ is bounded by the original in that domain plus what can enter from $\Omega_u$, the sum being the energy in the full domain, $\Omega$.
\end{rem}

In summary, we have proved
\begin{thm}
The overset domain problems defined in \eqref{eq:L-RSystem} are energy bounded with bounds \eqref{eq:Sharp1DScalarEstimate}.
\label{thm:ScalarEnergyBondedness}
\end{thm}

We use Thm. \ref{thm:ScalarEnergyBondedness}, to prove
\begin{thm}
The overset domain problem \eqref{eq:L-RSystem} is well-posed and its solution is equivalent to the solution of the original problem \eqref{eq:ScalarOmegaBVP} for $\omega$ on $\Omega$.
\label{thm:WellPosednessScalar1D}
\end{thm}
\begin{proof}
When we subtract \eqref{eq:ScalarOmegaBVP} from \eqref{eq:L-RSystem},
\begin{equation}
L'\;\left\{\begin{gathered}
u'_t + \alpha u'_x = 0, x\in\Omega_u\hfill\\
u'(a,t) = 0\hfill\\
u'(x,0) = 0 \hfill
\end{gathered} \right. \quad
R'\;\left\{\begin{gathered}
v'_t + \alpha v'_x = 0, x\in\Omega_v\hfill\\
v'(b,t) = u'(b,t)\hfill\\
v'(x,0) = 0 \hfill
\end{gathered} \right.
\label{eq:L-RSystemPerturbation}
\end{equation}
where $u' = u - \omega$ and $v'=v - \omega$. The problems $L'$ and $R'$ are identical to $L$ and $R$ except for the initial conditions. Then by \eqref{eq:Sharp1DScalarEstimate},
\begin{equation}
\begin{split}
& \inorm{u'}_{\Omega_u}^2 \le \inorm{u'(0)}_{\Omega_u}^2 = 0,
\\&
\inorm{v'}_{\Omega_v}^2 \le \inorm{v'(0)}_{\Omega_v}^2 +  \inorm{u'(0)}_{\Omega_{\bar u}}^2 = 0.
\end{split}
\label{eq:Sharp1DScalarEstimatePerturbation}
\end{equation}
Therefore the differences between the overset solutions and the full domain solution are zero, showing that $u = \omega$ and $v = \omega$, and that the overset domain problem provides the solution to the original single domain problem. Since the original problem for $\omega$ is well-posed, the overset domain problem is well-posed, too.
\end{proof}

\subsection{The Symmetric System in 1D}

We now extend \eqref{eq:ScalarOmegaBVP} to the system
\begin{equation}
\left\{\begin{gathered}
\statevecGreek\omega_t + \mmatrix A\statevecGreek\omega_x = 0, x\in\Omega\hfill\\
\statevecGreek\omega(a,t) = BL_1\left( \statevecGreek\omega(a,t)\right)\hfill\\
\statevecGreek\omega(d,t) = BR_2\left( \statevecGreek\omega(d,t)\right)\hfill\\
\statevecGreek\omega(x,0) = \statevecGreek \omega_0(x) \hfill
\end{gathered} \right.
\label{eq:SystemForOmega1D}
\end{equation}
We assume for convenience that $\mmatrix A$ is symmetric, i.e., $\mmatrix A = \mmatrix A^T$, and constant, and that the system is hyperbolic so that $\mmatrix A = \mmatrix P \Lambda \mmatrix P^{-1}$. With $\mmatrix A$ symmetric, $\mmatrix P^{-1}= \mmatrix P^{T}$. We also assume that there are no zero eigenvalues, although this is also not essential and does not change the results. 

The boundary operators, $BL_1, BR_2$ are linear, characteristic, and set the incoming characteristic states at the physical boundaries to zero. 
Since the problem \eqref{eq:SystemForOmega1D} is hyperbolic, we can define the characteristic variables
\begin{equation} 
\statevec w = \mmatrix P^{-1}\statevecGreek \omega = \left[\begin{array}{c}\statevec w^{+} \\\statevec w^{-}\end{array}\right],
\label{eq:wDefinition}
\end{equation}
where $\statevec w^{+}$ is associated with the positive eigenvalues of $\mmatrix A$ and $\statevec w^{-}$ is associated with the negative ones. Let us also write
 \begin{equation}
 \Lambda = \left[\begin{array}{cc}\bar \Lambda^+ & 0 \\0 & \bar \Lambda^-\end{array}\right],\quad
  \Lambda^{+} = \left[\begin{array}{cc}\bar \Lambda^+ & 0 \\0 & 0\end{array}\right],\quad \Lambda^{-} = \left[\begin{array}{cc}0 & 0 \\0 & \bar \Lambda^-\end{array}\right].
 \end{equation} 
 Finally, we define $\mmatrix A^\pm = \oneHalf \left\{ \mmatrix A \pm |\mmatrix A|\right\}$ so that $\mmatrix A = \mmatrix A^+ + \mmatrix A^-$, where $\mmatrix A^+>0$ and $\mmatrix A^-<0$.
In terms of the characteristic variables, we have the boundary operators $BL_1$ and $BR_2$ to implement the characteristic conditions
\begin{equation}
\statevec w^+(a,t) = \statevec 0
,\quad
\statevec w^-(d,t) = \statevec 0,
\end{equation}
in the forms
\begin{equation}
BL_1\left( \statevecGreek \omega(a,t)\right) = \mmatrix P \left[\begin{array}{c}\statevec 0 \\\statevec w^{-}(a,t)\end{array}\right] 
,\quad
BR_2\left( \statevecGreek \omega(d,t)\right) = \mmatrix P \left[\begin{array}{c}\statevec w^{+}(d,t) \\ \statevec 0\end{array}\right].
\label{eq:System1DBCs}
\end{equation}

We extend the associated overset domain problem \eqref{eq:L-RSystem} to the system
\begin{equation}
LS\;\left\{\begin{gathered}
\statevec u_t + \mmatrix A \statevec u_x = 0, x\in\Omega_u\hfill\\
\statevec u(a,t) = BL_1\left( \statevec u(a,t)\right)\hfill\\
\statevec u(c,t) = BL_2\left( \statevec u(c,t), \statevec v(c,t)\right)\hfill\\
\statevec u(x,0) = \statevecGreek \omega_0(x) \hfill
\end{gathered} \right. \quad
RS\;\left\{\begin{gathered}
\statevec v_t + \mmatrix A \statevec v_x = 0, x\in\Omega_v\hfill\\
\statevec v(b,t) = BR_1\left( \statevec u(b,t),\statevec v(b,t)\right)\hfill\\
\statevec v(d,t) = BR_2\left( \statevec v(d,t)\right)\hfill\\
\statevec v(x,0) = \statevecGreek \omega_0(x) \hfill
\end{gathered} \right.
\label{eq:LS-RSSystem}
\end{equation}
for the same domains shown in Fig. \ref{fig:Overlap1D}. 
The physical boundary conditions are the same as for the full domain problem. The interface operators $BL_2$ and $BR_1$ implement the characteristic coupling conditions
\begin{equation}
\statevec w^-_u(c,t) = \statevec w^-_v(c,t)
,\quad
\statevec w^+_v(b,t) = \statevec w^+_u(b,t)
\label{eq:1DCharCoupling}
\end{equation}
in the forms
\begin{equation}
BL_2\left( \statevec u(c,t), \statevec v(c,t)\right) =\mmatrix P \left[\begin{array}{c}\statevec w_u^{+}(c,t) \\\statevec w^-_v(c,t)\end{array}\right] 
,\quad
BR_1\left( \statevec u(b,t), \statevec v(b,t)\right) =\mmatrix P \left[\begin{array}{c}\statevec w_u^{+}(b,t) \\\statevec w^-_v(b,t)\end{array}\right],
\label{eq:System1InterfaceCs}
\end{equation}
where $\statevec w_u$ is given by \eqref{eq:wDefinition}, with $\statevec u$ replacing $\statevecGreek\omega$, and similarly for $\statevec w_v$.

The proof of energy boundedness for the overset problem \eqref{eq:LS-RSSystem} will use the following energy statements:
\begin{lem}
Over any domain $[\ell,r]$, the energy formula for $\statevec q = \statevec u \text{ or } \statevec v$
\begin{equation}
\frac{d}{dt}\inorm{\statevec q}_{[\ell,r]}^2 = 
 \left.\statevec q^T\mmatrix A^+\statevec q\right|_\ell
-\left.\statevec q^T\left|\mmatrix A^-\right|\statevec q\right|_\ell 
- \left.\statevec q^T\mmatrix A^+\statevec q\right|_r 
+\left.\statevec q^T\left|\mmatrix A^-\right|\statevec q\right|_r 
\label{eq:GeneralEnergyStatement}
\end{equation}
holds, where $\inorm{\statevec q}_{[\ell,r]}^2 \equiv \int_\ell^r \statevec q^T \statevec q dx$.
\label{lem:NormOnQEquality}
\end{lem}
\begin{proof}
Multiplying the PDE by $\statevec q^T$ and integrating over the interval $[\ell,r]$,
\begin{equation}
\int_\ell^r \statevec q^T \statevec q_t dx + \int_\ell^r \statevec q^T  \mmatrix A \statevec q_x dx = 0.
\end{equation}
Integrating the second integral by parts, and using the fact that the coefficient matrix is symmetric,
\begin{equation}
\oneHalf \frac{d}{dt}\inorm{\statevec q}^2 + \oneHalf \left. \statevec q^T  \mmatrix A \statevec q\right|_\ell^r = 0.
\end{equation}
But $\mmatrix A = \mmatrix A^+ + \mmatrix A^-$ and $\mmatrix A^- = -| \mmatrix A^-|$, so
\begin{equation}
 \frac{d}{dt}\inorm{\statevec q}^2 =-  \left. \statevec q^T  \mmatrix A^+ \statevec q\right|_\ell^r + \left. \statevec q^T  |\mmatrix A^-| \statevec q\right|_\ell^r,
\end{equation}
from which the result follows.
\end{proof}

Note that $\statevec q^T\mmatrix A^\pm\statevec q = \statevec w^{\pm,T}\bar \Lambda^\pm  \statevec w^{\pm}$. Therefore, as a corollary to Lemma \ref{lem:NormOnQEquality}, we have
\begin{cor}
The equality \eqref{eq:GeneralEnergyStatement} can be written in terms of the characteristic variables as
\begin{equation}
\frac{d}{dt}\inorm{\statevec q}_{[\ell,r]}^2 = 
 \left.\statevec w^{+,T}\bar \Lambda^+\statevec w^+\right|_\ell 
-\left.\statevec w^{-,T}\left| \bar \Lambda^-\right| \statevec w^-\right|_\ell 
- \left.\statevec w^{+,T}\bar \Lambda^+\statevec w^+\right|_r 
+\left.\statevec w^{-,T}\left| \bar \Lambda^-\right| \statevec w^-\right|_r .
\label{eq:GeneralEnergyStatementChar}
\end{equation}
\label{cor:Corollary1}
\end{cor}
So, for instance, the first and last terms on the right of \eqref{eq:GeneralEnergyStatement} and \eqref{eq:GeneralEnergyStatementChar}
correspond to incoming characteristic information (to be specified by data), whereas the middle two terms correspond to energy dissipation through wave propagation out of the domain.
\subsubsection{Well-Posedness for Diagonalizable Systems}

The system problem \eqref{eq:LS-RSSystem} is energy bounded, which we show by exploiting the fact that the coefficient matrix is diagonalizable.
In one space dimension (but generally not in two or three), the systems of equations \eqref{eq:LS-RSSystem} can be decoupled into
\begin{equation}
\frac{d}{dt}\statevec w^\pm + \bar\Lambda^\pm \statevec w_x^\pm = 0,
\end{equation}
each of which can be treated as in the scalar problem. Therefore, when we insert the inflow charcteristic boundary condition on $\statevec w^-_u$ at $x=c$ and the trivial inflow condition on $\statevec w^+_u$ at $x=a$,
\begin{equation}
\begin{split}
&\frac{d}{dt}\inorm{\statevec w^+_u}_{\Omega_u}^2 = -\left. \statevec w_u^{+,T}\bar\Lambda^+ \statevec w_u^{+}\right|_c,
\\&
\frac{d}{dt}\inorm{\statevec w^-_u}_{\Omega_u}^2 = -\left. \statevec w_u^{-,T}\left| \bar\Lambda^-\right| \statevec w_u^{-}\right|_a + \statevec w_v^{-,T}\left. \left| \bar\Lambda^-\right| \statevec w_v^{-}\right|_c.
\end{split}
\end{equation}
We still have (T1), the outflow trick \eqref{eq:PartialBoundL} on $\statevec w^+_u$,
\begin{equation}
\left. \statevec w_u^{+,T}\bar\Lambda^+ \statevec w_u^{+}\right|_b = -\frac{d}{dt}\inorm{\statevec w^+_u}_{ \Omega_{\bar u}}^2,
\label{eq:WplusAtb}
\end{equation}
which couples the state at the interior point $x=b$ to the initial and (trivial) inflow boundary data.

For the problem on the right, RS in \eqref{eq:LS-RSSystem}, we have a similar decomposition,
\begin{equation}
\begin{split}
&\frac{d}{dt}\inorm{\statevec w^+_v}_{\Omega_v}^2 = \left. \statevec w_u^{+,T}\bar\Lambda^+ \statevec w_u^{+}\right|_b 
-\left.  \statevec w_v^{+,T}\bar\Lambda^+ \statevec w_v^{+}\right|_d,
\\&
\frac{d}{dt}\inorm{\statevec w^-_v}_{\Omega_v}^2 =   -\left. \statevec w_v^{-,T}\left|\bar\Lambda^-\right| \statevec w_v^{-}\right|_b,
\end{split}
\end{equation}
plus
\begin{equation}
\left. \statevec w_v^{-,T}\left|\bar\Lambda^- \right| \statevec w_v^{-}\right|_c = -\frac{d}{dt}\inorm{\statevec w^-_v}_{ \Omega_{\bar v}}^2.
\label{eq:WminusAtc}
\end{equation}

Therefore,
\begin{equation}
\begin{split}
&\frac{d}{dt}\inorm{\statevec w^+_u}_{\Omega_u}^2 \le 0,
\\&
\frac{d}{dt}\inorm{\statevec w^-_u}_{\Omega_u}^2\le  -\frac{d}{dt}\inorm{\statevec w^-_v}_{ \Omega_{\bar v}}^2
\end{split}
\label{eq:wubounds}
\end{equation}
and
\begin{equation}
\begin{split}
&\frac{d}{dt}\inorm{\statevec w^+_v}_{\Omega_v}^2 \le -\frac{d}{dt}\inorm{\statevec w^+_u}_{ \Omega_{\bar u}}^2,
\\&
\frac{d}{dt}\inorm{\statevec w^-_v}_{\Omega_v}^2 \le 0.
\end{split}
\label{eq:wvbounds}
\end{equation}

We add the contributions from \eqref{eq:wubounds} and \eqref{eq:wvbounds} to get the total energy. Note that \begin{equation}
    \inorm{\statevec w}^2 = \inorm{\statevec w^+}^2 + \inorm{\statevec w^-}^2.
\end{equation}
With $\mmatrix A$ symmetric, $\mmatrix P^{-1}= \mmatrix P^{T}$, so $\mmatrix P$ is unitary. Therefore, for $\statevec q = \statevec u\;\rm{or}\; \statevec v$, $\inorm{\statevec w}^2  = \inorm{\statevec q}^2$ and
\begin{equation}
\inorm{\statevec q}^2 = \inorm{\statevec w^+}^2 + \inorm{\statevec w^-}^2.
\label{eq:NormDecomp}
\end{equation}
Then when we add the parts in \eqref{eq:wubounds} and \eqref{eq:wvbounds},
\begin{equation}
\begin{split}
&\frac{d}{dt}\inorm{\statevec u}_{\Omega_u}^2+\frac{d}{dt}\inorm{\statevec w^-_v}_{ \Omega_{\bar v}}^2\le 0,
\\&
\frac{d}{dt}\inorm{\statevec v}_{\Omega_v}^2+ \frac{d}{dt}\inorm{\statevec w^+_u}_{ \Omega_{\bar u}}^2\le 0.
\end{split}
\label{eq:U&VsystemEnergy1}
\end{equation}
Integrating \eqref{eq:U&VsystemEnergy1} in time,
\begin{equation}
\begin{split}
&\inorm{\statevec u(T)}_{\Omega_u}^2+\inorm{\statevec w^-_v(T)}_{ \Omega_{\bar v}}^2\le \inorm{\statevec u(0)}_{\Omega_u}^2+\inorm{\statevec w^-_v(0)}_{ \Omega_{\bar v}}^2,
\\&
\inorm{\statevec v(T)}_{\Omega_v}^2+ \inorm{\statevec w^+_u(T)}_{ \Omega_{\bar u}}^2\le \inorm{\statevec v(0)}_{\Omega_v}^2+ \inorm{\statevec w^+_u(0)}_{ \Omega_{\bar u}}^2.
\end{split}
\label{eq:Sharpest1DsystemEstimate}
\end{equation}
Finally, using \eqref{eq:NormDecomp} again, 
\begin{equation}
\begin{split}
&\inorm{\statevec u(T)}_{\Omega_u}^2 \le \inorm{\statevec u(0)}_{\Omega_u}^2+\inorm{\statevec v(0)}_{ \Omega_{\bar v}}^2,
\\&
\inorm{\statevec v(T)}_{\Omega_v}^2 \le \inorm{\statevec v(0)}_{\Omega_v}^2+ \inorm{\statevec u(0)}_{ \Omega_{\bar u}}^2,
\end{split}
\label{eq:SharperSystemEstimate}
\end{equation}
and
\begin{equation}
\begin{split}
&\inorm{\statevec u(T)}_{\Omega_u}^2 \le \inorm{\statevecGreek\omega_0}_{\Omega_u}^2+\inorm{\statevecGreek\omega_0}_{ \Omega_{\bar v}}^2 = \inorm{\statevecGreek\omega_0}_{\Omega}^2,
\\&
\inorm{\statevec v(T)}_{\Omega_v}^2 \le \inorm{\statevecGreek\omega_0}_{\Omega_v}^2+ \inorm{\statevecGreek\omega_0}_{ \Omega_{\bar u}}^2 =  \inorm{\statevecGreek\omega_0}_{\Omega}^2.
\end{split}
\label{eq:EnergyBndnessWithDiagonalization}
\end{equation}

Therefore, we have proved
\begin{thm}
The overset domain problem \eqref{eq:LS-RSSystem} in $1\rm D$ with characteristic boundary \eqref{eq:System1DBCs} and coupling conditions \eqref{eq:1DCharCoupling} is energy bounded, with the energy satisfying \eqref{eq:EnergyBndnessWithDiagonalization}.
The sharpest estimate, \eqref{eq:Sharpest1DsystemEstimate}, which expresses (T2), shows that with decoupling, the energy for $\statevec u$ depends only on its initial data and the left-going energy from $\statevec v$, and similarly for $\statevec v$. 
\end{thm}

As in the scalar case, the overset system problem \eqref{eq:LS-RSSystem} is well-posed and provides the same solution as the original, full domain problem \eqref{eq:SystemForOmega1D}, which we state as

\begin{thm}
The solutions of the overset domain problem \eqref{eq:LS-RSSystem} match those of the original full domain problem, \eqref{eq:SystemForOmega1D}, i.e. $\statevec u = \statevecGreek\omega$ and $\statevec v = \statevecGreek\omega$. Under the assumption that the original problem is well-posed, the overset domain problem is well-posed.
\label{thm:1DSystemWellPosedness}
\end{thm}

\begin{proof}
The difference between the systems in \eqref{eq:LS-RSSystem} and \eqref{eq:SystemForOmega1D}, gives equations for the differences $\statevec u' = \statevec u- \statevecGreek \omega$ and $\statevec v' = \statevec v - \statevecGreek \omega$, each of the same form but with trivial initial conditions. Then by \eqref{eq:EnergyBndnessWithDiagonalization},
\begin{equation}
\begin{split}
&\inorm{\statevec u'(T)}_{\Omega_u}^2 \le \inorm{\statevec u'(0)}_{\Omega_u}^2+\inorm{\statevec v'(0)}_{ \Omega_{\bar v}}^2=0,
\\&
\inorm{\statevec v'(T)}_{\Omega_v}^2 \le \inorm{\statevec v'(0)}_{\Omega_v}^2+ \inorm{\statevec u'(0)}_{ \Omega_{\bar u}}^2=0
\end{split}
\label{eq:EnergyBndnessWithDiagonalizationDelta}
\end{equation}
and the result follows as in Thm. \ref{thm:WellPosednessScalar1D}.
\end{proof}

\subsubsection{Energy without Characteristic Decoupling}
The proof of energy boundedness to get \eqref{eq:EnergyBndnessWithDiagonalization} used the fact that the characteristic variables are independent in the interior of the domains, not just at the boundaries. Physically, the decoupling allows energy propagating along left-- and right--going waves to be separated so that only incoming wave energies are included at the overlap boundaries, through \eqref{eq:WplusAtb} and \eqref{eq:WminusAtc}. The idea does not extend to greater than one space dimension since the matrices are generally not simultaneously diagonalizable. In this section we show what the energy method gives when we can't rely on diagonalizability.

For the overset domain problem, with (trivial characteristic, $\statevec w = 0$) inflow physical boundary conditions applied, Lemma \ref{lem:NormOnQEquality} says that
\begin{equation}\begin{split}
&\frac{d}{dt}\inorm{\statevec u}_{\Omega_u}^2 
=  -\left.\statevec u^T\left|\mmatrix A^-\right|\statevec u\right|_a 
- \left.\statevec u^T\mmatrix A^+\statevec u\right|_c 
+\left.\statevec u^T\left|\mmatrix A^-\right|\statevec u\right|_c 
\\&
\frac{d}{dt}\inorm{\statevec v}_{\Omega_v}^2 
= 
-\left.\statevec v^T\left|\mmatrix A^-\right|\statevec v\right|_b 
+\left.\statevec v^T\mmatrix A^+\statevec v\right|_b - \left.\statevec v^T\mmatrix A^+\statevec v\right|_d .
\end{split}
\end{equation}
We then impose the characteristic interface conditions and re-order terms
\begin{equation}\begin{split}
&\frac{d}{dt}\inorm{\statevec u}_{\Omega_u}^2 
=  -\left.\statevec u^T\left|\mmatrix A^-\right|\statevec u\right|_a 
- \left.\statevec u^T\mmatrix A^+\statevec u\right|_c 
+\left.\statevec v^T\left|\mmatrix A^-\right|\statevec v\right|_c 
\\&
\frac{d}{dt}\inorm{\statevec v}_{\Omega_v}^2 
=
-\left.\statevec v^T\left|\mmatrix A^-\right|\statevec v\right|_b 
- \left.\statevec v^T\mmatrix A^+\statevec v\right|_d  
+\left.\statevec u^T\mmatrix A^+\statevec u\right|_b .
\end{split}
\label{eq:EnergyInEachDomain}
\end{equation}

The first two terms in each of \eqref{eq:EnergyInEachDomain} are dissipative and represent loss of energy at the boundaries due to waves leaving the domain. The final term in each is due to the incoming information from the alternate domain. 

To get the energy from the complementary domains, we use \eqref{eq:GeneralEnergyStatement} again,
\begin{equation}
\frac{d}{dt}\inorm{\statevec u}_{\Omega_{\bar u}}^2 = 
\left.\statevec u^T\mmatrix A^+\statevec u\right|_a 
-\left.\statevec u^T\left|\mmatrix A^-\right|\statevec u\right|_a 
- \left.\statevec u^T\mmatrix A^+\statevec u\right|_b 
+\left.\statevec u^T\left|\mmatrix A^-\right|\statevec u\right|_b.
\label{eq:GeneralEnergyStatement2}
\end{equation}
Then after we apply the boundary condition $\statevec w_u = 0$ at $x=a$, the interface point value is
\begin{equation}
\left.\statevec u^T\mmatrix A^+\statevec u\right|_b  = 
-\frac{d}{dt}\inorm{\statevec u}_{\Omega_{\bar u}}^2
-\left.\statevec u^T\left|\mmatrix A^-\right|\statevec u\right|_a 
+\left.\statevec u^T\left|\mmatrix A^-\right|\statevec u\right|_b.
\label{eq:GeneralEnergyStatement3}
\end{equation}
Similarly,
\begin{equation}
\frac{d}{dt}\inorm{\statevec v}_{\Omega_{\bar v}}^2 = 
\left.\statevec v^T\mmatrix A^+\statevec v\right|_c 
-\left.\statevec v^T\left|\mmatrix A^-\right|\statevec v\right|_c
- \left.\statevec v^T\mmatrix A^+\statevec v\right|_d 
+\left.\statevec v^T\left|\mmatrix A^-\right|\statevec v\right|_d,
\label{eq:GeneralEnergyStatement4}
\end{equation}
so that after applying the boundary condition at $x = d$, the interface point value is
\begin{equation}
\left.\statevec v^T\left|\mmatrix A^-\right|\statevec v\right|_c = 
-\frac{d}{dt}\inorm{\statevec v}_{\Omega_{\bar v}}^2 
+\left.\statevec v^T\mmatrix A^+\statevec v\right|_c 
- \left.\statevec v^T\mmatrix A^+\statevec v\right|_d .
\label{eq:GeneralEnergyStatement5}
\end{equation}

Finally, we add the two equations in \eqref{eq:EnergyInEachDomain}, replace the terms from \eqref{eq:GeneralEnergyStatement3} and \eqref{eq:GeneralEnergyStatement5}, and use the fact that the outflow terms are dissipative to get the bound
\begin{equation}
\frac{d}{dt}\inorm{\statevec u}_{\Omega_{u}}^2 + \frac{d}{dt}\inorm{\statevec v}_{\Omega_{v}}^2 + \frac{d}{dt}\inorm{\statevec u}_{\Omega_{\bar u}}^2 + \frac{d}{dt}\inorm{\statevec v}_{\Omega_{\bar v}}^2\le
\left\{ \statevec v^T \mmatrix A^+ \statevec v - \statevec u^T\mmatrix A^+\statevec u\right\}_c + \left\{ \statevec u^T| \mmatrix A^-| \statevec u - \statevec v^T| \mmatrix A^-|\statevec v\right\}_b.
\label{eq:TotalEnergyDerivSystem}
\end{equation}
The bound has two positive terms, $\left.\statevec v^T \mmatrix A^+ \statevec v\right|_c$ and $\left.\statevec u^T| \mmatrix A^-| \statevec u\right|_b$, for which there is no initial or boundary data, so without diagonalization the energy method does not give an energy bound in terms of data.
When we integrate \eqref{eq:TotalEnergyDerivSystem} in time and use the inequality associated with the dissipative terms, we get the bound
\begin{equation}
\begin{split}
\inorm{\statevec u(T)}_{\Omega_{u}}^2 + \inorm{\statevec v(T)}_{\Omega_{v}}^2 + \inorm{\statevec u(T)}_{\Omega_{\bar u}}^2 + \inorm{\statevec v(T)}_{\Omega_{\bar v}}^2
\le&
\inorm{\statevec u_0}_{\Omega_{u}}^2 + \inorm{\statevec v_0}_{\Omega_{v}}^2 + \inorm{\statevec u_0}_{\Omega_{\bar u}}^2 + \inorm{\statevec v_0}_{\Omega_{\bar v}}^2 
\\&+ \int_0^T \left\{\statevec v^T \mmatrix A^+ \statevec v |_c+ \statevec u^T| \mmatrix A^-| \statevec u|_b\right\}dt.
\end{split}
\label{eq:TotalEnergySystem}
\end{equation}

\begin{rem}Overset techniques have been used for many years in multiple space dimensions where diagonalization is not possible, so one might ask why numerical schemes should work at all. Looking at \eqref{eq:TotalEnergyDerivSystem}, we see that in practice, with the two solutions being close to each other (differing, perhaps, by an amount depending on resolution and truncation error), the terms on the right of \eqref{eq:TotalEnergyDerivSystem},
\begin{equation}
\left\{ \statevec v^T \mmatrix A^+ \statevec v - \statevec u^T\mmatrix A^+\statevec u\right\}_c + \left\{ \statevec u^T| \mmatrix A^-| \statevec u - \statevec v^T| \mmatrix A^-|\statevec v\right\}_b,
\label{eq:UnwantedTerms}
\end{equation}
will likely be small, making the right hand side of \eqref{eq:TotalEnergySystem} overly pessimistic, and small enough that numerical dissipation elsewhere could dominate those terms.
\end{rem}
% Thus, the norm of the solution in not bounded in terms of data, since the time integral interface term has no data associated with it. Note there is no maximum principle of hyperbolic systems, so we cannot bound the internal terms, as we could the scalar case.

The result \eqref{eq:TotalEnergySystem} differs from that of the previous section, \eqref{eq:EnergyBndnessWithDiagonalization}, in that with decoupling it is possible to exclude interior and subsequently boundary contributions that cannot be bounded by data. 

\subsection{An Energy Bounded Overset Problem with a Weighted Multiple Penalty}\label{sec:WeightedMultiplePenalty}

The energy method did not show energy boundedness for the system problem without taking diagonalization into account. We now show that we can eliminate the unwanted terms on the right of \eqref{eq:TotalEnergySystem} by applying penalties in the domains, (T3), and pose a new energy bounded problem that does not require diagonalizability in the proof.

To eliminate the extra terms that arise in the interior of the domains, we apply the third technique, (T3), and impose additional interface conditions as penalty terms:
\begin{equation}
LS\;\left\{\begin{gathered}
\statevec u_t + \mmatrix A\statevec u_x + \mathcal L_u^b \left[ \mmatrix{$\Sigma$}_u^b (\statevec u-\statevec v)\right] + \mathcal L_u^c \left[ \mmatrix{$\Sigma$}_u^c (\statevec u-\statevec v)\right]= 0,\quad  x\in\Omega_u\hfill\\
\statevec u(a,t) = BL_1\left( \statevec u(a,t)\right)\hfill\\
\statevec u(x,0) = \statevecGreek \omega_0(x) \hfill
\end{gathered} \right.
\label{eq:LS-RSSystem-MPu}
\end{equation}
\begin{equation}
RS\;\left\{\begin{gathered}
\statevec v_t + \mmatrix A \statevec v_x + \mathcal L_v^b \left[  \mmatrix{$\Sigma$}_v^b (\statevec v-\statevec u)\right]+ \mathcal L_v^c \left[  \mmatrix{$\Sigma$}_v^c (\statevec v-\statevec u)\right]= 0,\quad x\in\Omega_v\hfill\\
\statevec v(d,t) = BR_2\left( \statevec v(d,t)\right)\hfill\\
\statevec v(x,0) = \statevecGreek \omega_0(x) \hfill
\end{gathered} \right.
\label{eq:LS-RSSystem-MPv}
\end{equation}
where $\mathcal L_{u}^b$ is a lifting operator \cite{Arnoldetal2002} that integrates over an interval $I$ containing $x=b$ as
\begin{equation}
    \int_I \statevecGreek\psi^T\mathcal L_u^b[\statevecGreek\phi]dx = \left.\statevecGreek\psi^T\statevecGreek\phi\right|_b,
\end{equation}
etc., and $\statevecGreek\psi$ is a test function. We must find specific conditions on the coupling matrices $\mmatrix{$\Sigma$}_u$ and $\mmatrix{$\Sigma$}_v$, so that the problem \eqref{eq:LS-RSSystem-MPu}--\eqref{eq:LS-RSSystem-MPv} is energy bounded and conservative.

\begin{rem}
The penalty at $x=c$ in the problem $LS$ and the penalty at $x=b$ in $RS$ are weak weighted enforcements of the (coupling) boundary conditions at those points.
% \begin{equation}
% \begin{gathered}
%     \left.\mmatrix{$\Sigma$}_u^c \statevec u\right|_c=\left.\mmatrix{$\Sigma$}_u^c\statevec v\right|_c
% ,\hfill\\
%     \left.\mmatrix{$\Sigma$}_v^b \statevec v\right|_b=\left.\mmatrix{$\Sigma$}_v^b\statevec u\right|_b.
% \end{gathered}
% \end{equation}
\end{rem}

We first create weak forms of the PDEs in  \eqref{eq:LS-RSSystem-MPu}--\eqref{eq:LS-RSSystem-MPv} by multiplying by test functions $\statevecGreek\phi_u$ and $\statevecGreek\phi_v$, and then integrating over the two domains
\begin{equation}
\begin{gathered}
 \iprod{\statevecGreek\phi_u,\statevec u_t}_{\Omega_u} + \iprod{\statevecGreek\phi_u,\mmatrix A u_x }_{\Omega_u}+
 \left.\statevecGreek\phi_u^T\ \mmatrix{$\Sigma$}_u^b (\statevec u-\statevec v)\right|_b+
 \left.\statevecGreek\phi_u^T\ \mmatrix{$\Sigma$}_u^c (\statevec u-\statevec v)\right|_c= 0,\hfill\\
\iprod{\statevecGreek\phi_v,\statevec v_t}_{\Omega_v} + \iprod{\statevecGreek\phi_v,\mmatrix A \statevec v_x}_{\Omega_v} +  \left.\statevecGreek\phi_v^T \mmatrix{$\Sigma$}_v^b (\statevec v-\statevec u)\right|_b+   \left.\statevecGreek\phi_v^T \mmatrix{$\Sigma$}_v^c (\statevec v-\statevec u)\right|_c= 0.\hfill\\
\end{gathered}
\label{eq:WeakFormsFonFullDomains}
\end{equation}
To implement (T1), the overlap region has the weak forms
\begin{equation}
\begin{gathered}
 \iprod{\statevecGreek\phi_u,\statevec u_t}_{\Omega_O} + \iprod{\statevecGreek\phi_u,\mmatrix A \statevec u_x }_{\Omega_O}= 0\hfill\\
\iprod{\statevecGreek\phi_v,\statevec v_t}_{\Omega_O} + \iprod{\statevecGreek\phi_v,\mmatrix A \statevec v_x}_{\Omega_O} = 0.\hfill\\
\end{gathered}
\label{eq:WeakFormsFonOverlaps}
\end{equation}
Finally, we integrate the space derivative inner products in \eqref{eq:WeakFormsFonFullDomains} and \eqref{eq:WeakFormsFonOverlaps} by parts,
\begin{equation}
\begin{gathered}
 \iprod{\statevecGreek\phi_u,\statevec u_t}_{\Omega_u} 
 + \left.\statevecGreek\phi_u^T\mmatrix A \statevec u \right|_a^c 
 - \iprod{\statevecGreek\phi'_u,\mmatrix A \statevec u }_{\Omega_u}
 +\left.\statevecGreek\phi_u^T\ \mmatrix{$\Sigma$}_u^b (\statevec u-\statevec v)\right|_b+\left.\statevecGreek\phi_u^T\ \mmatrix{$\Sigma$}_u^c (\statevec u-\statevec v)\right|_c= 0\hfill\\
\iprod{\statevecGreek\phi_v,\statevec v_t}_{\Omega_v} +  \left.\statevecGreek\phi_v^T\mmatrix A \statevec v \right|_b^d 
- \iprod{\statevecGreek\phi'_v,\mmatrix A \statevec v}_{\Omega_v} +\left.\statevecGreek\phi_v^T  \mmatrix{$\Sigma$}_v^b (\statevec v-\statevec u)\right|_b+  \left.\statevecGreek\phi_v^T \mmatrix{$\Sigma$}_v^c (\statevec v-\statevec u)\right|_c= 0,\hfill\\
\end{gathered}
\label{eq:WeakFormsUV}
\end{equation}
and
\begin{equation}
\begin{gathered}
 \iprod{\statevecGreek\phi_u,\statevec u_t}_{\Omega_O} + \left.\statevecGreek\phi_u^T\mmatrix A \statevec u \right|_b^c - \iprod{\statevecGreek\phi'_u,\mmatrix A \statevec u}_{\Omega_O}= 0\hfill\\
\iprod{\statevecGreek\phi_v,\statevec v_t}_{\Omega_O} + \left.\statevecGreek\phi_v^T\mmatrix A \statevec v \right|_b^c - \iprod{\statevecGreek\phi'_v,\mmatrix A \statevec v}_{\Omega_O} = 0,\hfill\\
\end{gathered}
\label{eq:WeakFormsOverlap}
\end{equation}
where the prime represents the partial derivative with respect to $x$. Technically speaking, we apply the lifting operator $\mathcal L^b$ at $b - \epsilon$ and $\mathcal L^c$ at $c + \epsilon$. After integration we get \eqref{eq:WeakFormsUV} and \eqref{eq:WeakFormsOverlap} at those shifted points. We then take the limit as $\epsilon\rightarrow 0$, so the penalty does not appear in the overlap region.

\subsubsection{Energy Boundedness}\label{sec:1DEnergyBoundedness}
With specific conditions on the coupling matrices, $\mmatrix{$\Sigma$}_u$ and $\mmatrix{$\Sigma$}_u$, the overset problem
\eqref{eq:LS-RSSystem-MPu} and \eqref{eq:LS-RSSystem-MPv} is energy bounded. To show that, we set $\statevecGreek\phi_u = \statevec u$ and $\statevecGreek\phi_v = \statevec v$ in \eqref{eq:WeakFormsUV}--\eqref{eq:WeakFormsOverlap} to construct the energy in each domain. Since for any $\statevec q$, $\iprod{\statevec q,\mmatrix A \statevec q_x}_{[\ell,r]} = \oneHalf \statevec q^T A \statevec q|_{\ell}^r$,
\begin{equation}
\begin{gathered}
\frac{d}{dt}\inorm{\statevec u}_{\Omega_u}^2 + \left.\statevec u^T \mmatrix A \statevec u\right|_a^c + \left.2\statevec u^T \mmatrix{$\Sigma$}_u^b(\statevec  u - \statevec v)\right|_b + \left.2\statevec u^T\mmatrix{$\Sigma$}_u^c(\statevec  u - \statevec v)\right|_c= 0, \hfill\\
\frac{d}{dt}\inorm{\statevec v}_{\Omega_v}^2 + \left.\statevec v^T \mmatrix A \statevec v\right|_b^d 
+ \left.2\statevec v^T \mmatrix{$\Sigma$}_v^b(\statevec  v - \statevec u)\right|_b
+ \left.2\statevec v^T\mmatrix{$\Sigma$}_v^c(\statevec  v - \statevec u)\right|_c = 0.
\hfill\end{gathered}
\label{eq:OneDFullNormDerivativesW}
\end{equation}
We also have in the overlap region, $\Omega_O$, that
\begin{equation}
\begin{gathered}
\frac{d}{dt}\inorm{\statevec u}_{\Omega_O}^2 + \left.\statevec u^T \mmatrix A \statevec u\right|_b^c = 0 \\
\frac{d}{dt}\inorm{\statevec v}_{\Omega_O}^2 + \left.\statevec v^T \mmatrix A \statevec v\right|_b^c  = 0.
\end{gathered}
\label{eq:OneDOverlapNormDerivativesW}
\end{equation}

Now we note the following relationships between the norms in the various regions,
\begin{equation}
\begin{gathered}
\inorm{\statevec u}^2_{\Omega_u} = \inorm{\statevec u}^2_{\Omega_{\bar u}} + \inorm{\statevec u}^2_{\Omega_O},\\
\inorm{\statevec v}^2_{\Omega_v} = \inorm{\statevec v}^2_{\Omega_{\bar v}} + \inorm{\statevec v}^2_{\Omega_O}.\\
\end{gathered}
\end{equation}
Then it is also true that, taking $0< \eta< 1$,
\begin{equation}
\begin{gathered}
\inorm{\statevec u}^2_{\Omega_u} - \eta \ \inorm{\statevec u}^2_{\Omega_O} = \inorm{\statevec u}^2_{\Omega_{\bar u}} + (1-\eta)\inorm{\statevec u}^2_{\Omega_O}\ge 0,\\
\inorm{\statevec v}^2_{\Omega_v} - (1-\eta)  \inorm{\statevec v}^2_{\Omega_O} = \inorm{\statevec v}^2_{\Omega_{\bar v}} + \eta \inorm{\statevec v}^2_{\Omega_O}\ge 0\\
\end{gathered}
\label{eq:NewNorms1DW}
\end{equation}
define norms. From this observation we prove
\begin{lem}
For $0<\eta<1$, 
\begin{equation}
\begin{split}
E(\statevec u,\statevec v) 
&\equiv \inorm{\statevec u}^2_{\Omega_u} + \inorm{\statevec v}^2_{\Omega_v} 
 -  \left\{ \eta \inorm{\statevec u}^2_{\Omega_O} +  (1-\eta)\inorm{\statevec v}^2_{\Omega_O}\right\}
\\&
=\inorm{\statevec u}^2_{\Omega_{\bar u}} + (1-\eta)\inorm{\statevec u}^2_{\Omega_O} + \inorm{\statevec v}^2_{\Omega_{\bar v}} + \eta \inorm{\statevec v}^2_{\Omega_O}\ge 0
\end{split}
\label{eq:EnergyNorm}
\end{equation}
defines a norm that can be used to bound both $\statevec u$ and $\statevec v$. 
\label{lem:EBoundW}
\end{lem}

\begin{proof}
The energy, $E$, is a convex linear combination of norms, and so with $E\ge 0$, it satisfies the usual requirements for a norm. Furthermore, if $E(\statevec u,\statevec v)\le C$ then
\begin{equation}
\begin{split}
E(\statevec u,\statevec v) &=  \inorm{\statevec u}^2_{\Omega_u}  - \eta \inorm{\statevec u}^2_{\Omega_O} + \inorm{\statevec v}^2_{\Omega_v} - (1-\eta)  \inorm{\statevec v}^2_{\Omega_O}
\\&
= (1-\eta)\inorm{\statevec u}^2_{\Omega_u} + \eta\left(\inorm{\statevec u}^2_{\Omega_u}  -  \inorm{\statevec u}^2_{\Omega_O} \right) 
+  \eta\inorm{\statevec v}^2_{\Omega_v} + (1-\eta)\left(\inorm{\statevec v}^2_{\Omega_v} -   \inorm{\statevec v}^2_{\Omega_O}\right)
\\&
= (1-\eta)\inorm{\statevec u}^2_{\Omega_u} + \eta\inorm{\statevec v}^2_{\Omega_v} + \left\{\eta\inorm{\statevec u}^2_{\Omega_{\bar u}} +  (1-\eta)\inorm{\statevec v}^2_{\Omega_{\bar v}} \right\}
\\& 
\le C.
\end{split}
\label{ew:newNormBoundDef1DW}
\end{equation}
Since each term is non-negative, each is bounded by $C$, from which it follows that $\inorm{\statevec u}^2_{\Omega_{u}} \le C/(1-\eta)$ and $\inorm{\statevec v}^2_{\Omega_{v}} \le C/\eta$.
\end{proof}
\begin{rem}
The norms of $\statevec u$ and $\statevec v$ represent the energy in each subdomain. Adding them gives the energy over the original domain $\Omega$ but double counts the overlap region. The energy $E(\statevec u,\statevec v)$ subtracts the extra contribution as a weighted sum of the overlap energy to represent the energy of the original problem for $\statevecGreek\omega$ (T2). From the PDE point of view, the choice of $\eta\in (0,1)$ is arbitrary, but may have relevance numerically when one solution may be more accurate than the other due to increased resolution, for instance.
\end{rem}

Motivated by \eqref{eq:NewNorms1DW} and \eqref{ew:newNormBoundDef1DW}, we combine \eqref{eq:OneDFullNormDerivativesW} and \eqref{eq:OneDOverlapNormDerivativesW} as in \eqref{eq:EnergyNorm}, and use the boundary conditions at $x=a$ and $x=d$, to get
\begin{equation}
\begin{split}
&\frac{d}{dt}\inorm{\statevec u}_{\Omega_u}^2 + \frac{d}{dt}\inorm{\statevec v}_{\Omega_v}^2  - \left\{ \eta \frac{d}{dt}\inorm{\statevec u}_{\Omega_{O}}^2 +
 (1-\eta)\frac{d}{dt}\inorm{\statevec v}_{\Omega_{O}}^2 \right\}
\\ &
+ \left\{(1-\eta) \statevec u^T A \statevec u - (1-\eta) \statevec v^T A \statevec v
+ 2\statevec u^T\ \mmatrix{$\Sigma$}_u^c(\statevec u - \statevec v) + 2\statevec v^T \mmatrix{$\Sigma$}_v^c(\statevec v - \statevec u)\right\}_c
\\&
+ \left\{  \eta \statevec u^T A \statevec u - \eta \statevec v^T A \statevec v  +
2\statevec u^T\ \mmatrix{$\Sigma$}_u^b(\statevec u - \statevec v) + 2\statevec v^T \mmatrix{$\Sigma$}_v^b(\statevec v - \statevec u)\right\}_b
\\&
\le0.
\end{split}
\label{eq:EnergyTimeDerivW}
\end{equation}
Let us write \eqref{eq:EnergyTimeDerivW} as
\begin{equation}
\frac{d}{dt} E + \mathcal P_b + \mathcal P_c
\le0,
\end{equation}
where 
\begin{equation}\begin{gathered}
\mathcal P_b = \left\{  \eta \statevec u^T A \statevec u - \eta \statevec v^T A \statevec v  +
2\statevec u^T\ \mmatrix{$\Sigma$}_u^b(\statevec u - \statevec v) + 2\statevec v^T \mmatrix{$\Sigma$}_v^b(\statevec v - \statevec u)\right\}_b ,\\
\mathcal P_c = \left\{(1-\eta) \statevec u^T A \statevec u - (1-\eta) \statevec v^T A \statevec v
+ 2\statevec u^T\ \mmatrix{$\Sigma$}_u^c(\statevec u - \statevec v) + 2\statevec v^T \mmatrix{$\Sigma$}_v^c(\statevec v - \statevec u)\right\}_c.
\end{gathered}
\label{eq:Pin1D}
\end{equation}
A sufficient condition for the overset domain problem to be energy bounded, then, is $\mathcal P_b, \mathcal P_c \ge 0$.

We can re-write $\mathcal P_{b,c}$ as
\begin{equation}
\mathcal P_{b,c} = \left. \left[\begin{array}{c}\statevec u \\\statevec v\end{array}\right]^T\mmatrix M_{b,c}  \left[\begin{array}{c}\statevec u \\\statevec v\end{array}\right] \right|_{b,c},
\end{equation}
where
\begin{equation}
\mmatrix M_b =\left[\begin{array}{cc}\eta A + 2\mmatrix{$\Sigma$}^b_u& -(\mmatrix{$\Sigma$}^b_u  + \mmatrix{$\Sigma$}^b_v)\\-(\mmatrix{$\Sigma$}^b_u+ \mmatrix{$\Sigma$}^b_v) & -\eta\mmatrix A + 2\mmatrix{$\Sigma$}^b_v \end{array}\right]
,\quad\mmatrix M_c =\left[\begin{array}{cc}(1-\eta) A + 2\mmatrix{$\Sigma$}^c_u& -(\mmatrix{$\Sigma$}^c_u  + \mmatrix{$\Sigma$}^c_v)\\-(\mmatrix{$\Sigma$}^c_u+ \mmatrix{$\Sigma$}^c_v) & -(1-\eta)\mmatrix A + 2\mmatrix{$\Sigma$}^c_v \end{array}\right].
\end{equation}
So, a sufficient condition for energy boundedness is that both $\mmatrix M_{b,c}\ge 0$. 
The matrices $\mmatrix M_{b,c}$ are of the same form,
\begin{equation}
\mmatrix M =\left[\begin{array}{cc}\beta A + 2\ \mmatrix{$\Sigma$}_u& -(\ \mmatrix{$\Sigma$}_u  +  \mmatrix{$\Sigma$}_v)\\-(\ \mmatrix{$\Sigma$}_u+  \mmatrix{$\Sigma$}_v) & -\beta\mmatrix A + 2 \mmatrix{$\Sigma$}_v \end{array}\right],
\end{equation}
where $\beta = \eta$ at $x=b$ and $\beta = (1-\eta)$ at $x=c$.

The conditions for $\mathcal P_b, \mathcal P_c \ge 0$ can be found by 
% noticing that $\mmatrix M$ can be made block diagonally dominant by choosing $\ \mmatrix{$\Sigma$}_u$ and $ \mmatrix{$\Sigma$}_v$ sufficiently large. In particular, $\mmatrix M\ge 0$ if
% The last two conditions \eqref{eq:1DSigmaMatrixConditionsW3} and \eqref{eq:1DSigmaMatrixConditionsW4} are dependent, given the second, \eqref{eq:1DSigmaMatrixConditionsW2}.
rotating the matrix $\mmatrix M$ using
\begin{equation}
 \left[\begin{array}{c}\statevec u \\\statevec v\end{array}\right] 
 = \oneHalf \left[\begin{array}{cc}I & I \\-I & I\end{array}\right] \left[\begin{array}{c}\statevec u - \statevec v \\\statevec u + \statevec v\end{array}\right] \equiv \mmatrix R \left[\begin{array}{c}\statevec u - \statevec v \\\statevec u + \statevec v\end{array}\right].
\end{equation}
Then
\begin{equation}
\mathcal P = \left[\begin{array}{c}\statevec u - \statevec v \\\statevec u +\statevec v\end{array}\right]^T\mmatrix R^T \mmatrix M \mmatrix R \left[\begin{array}{c}\statevec u - \statevec v \\\statevec u +\statevec v\end{array}\right] \equiv  \left[\begin{array}{c}\statevec u - \statevec v \\\statevec u +\statevec v\end{array}\right]^T\tilde{\mmatrix M}\left[\begin{array}{c}\statevec u - \statevec v \\\statevec u +\statevec v\end{array}\right].
\end{equation}
Therefore a sufficient condition for energy boundedness is $\tilde {\mmatrix M} \ge0$, where
\begin{equation}
\tilde {\mmatrix M} = \frac{1}{4}\left[\begin{array}{cc}4(\ \mmatrix{$\Sigma$}_u+ \mmatrix{$\Sigma$}_v) & 2\beta A+2(\ \mmatrix{$\Sigma$}_u- \mmatrix{$\Sigma$}_v) \\2\beta A+2(\ \mmatrix{$\Sigma$}_u- \mmatrix{$\Sigma$}_v) & 0\end{array}\right].
\end{equation}
Non-negativity is guaranteed if 
\begin{subequations}
\begin{align}
\beta A + \mmatrix{$\Sigma$}_u =  \mmatrix{$\Sigma$}_v ,
\label{eq:1DSigmaMatrixConditionsW2}
\\
2 \mmatrix{$\Sigma$}_v - \beta A \ge 0.
\label{eq:1DSigmaMatrixConditionsW3}
\end{align}
\label{eq:1DSigmaMatrixConditionsW}
\end{subequations}
The condition \eqref{eq:1DSigmaMatrixConditionsW2} zeros the off-diagonal blocks of $\tilde{\mmatrix M}$, leaving only the upper left block in the rotated matrix. Adding condition \eqref{eq:1DSigmaMatrixConditionsW2} to \eqref{eq:1DSigmaMatrixConditionsW3} shows that the upper left block is non-negative. With those conditions, 
\begin{equation}
\mathcal P = (\statevec u - \statevec v)^T (\ \mmatrix{$\Sigma$}_u +  \mmatrix{$\Sigma$}_v) (\statevec u - \statevec v)\ge 0.
\end{equation}

\begin{ex}
The upwind projection on the penalty  with $\eta = \beta = \oneHalf$ terms leads to an energy-bounded problem. 
Let $\ \mmatrix{$\Sigma$}_u = \oneHalf \left| A^-\right|$. Then since $A = A^+ - \left| A^-\right|$, condition \eqref{eq:1DSigmaMatrixConditionsW2} requires that $ \mmatrix{$\Sigma$}_v = \oneHalf A^+$. Then 
\begin{equation}
\mathcal P =  \oneHalf (\statevec u - \statevec v)^T \left(A^+ + \left| A^-\right|\right) (\statevec u - \statevec v)\ge 0.
\end{equation}
For this scenario, the equations to be solved for $\statevec u$ and $\statevec v$ are
\begin{equation}
\begin{gathered}
\statevec u_t + \mmatrix A u_x + \mathcal L_u^b \left[ \oneHalf \left| A^-\right| (\statevec u-\statevec v)\right] + \mathcal L_u^c \left[\oneHalf \left| A^-\right|  (\statevec u-\statevec v)\right]= 0,\quad  x\in\Omega_u\hfill\\
\statevec v_t + \mmatrix A \statevec v_x + \mathcal L_v^b \left[  \oneHalf A^+ (\statevec v-\statevec u)\right]+ \mathcal L_v^c \left[ \oneHalf A^+ (\statevec v-\statevec u)\right]= 0, \quad x\in\Omega_v.\hfill\\
\end{gathered}
\end{equation}
\end{ex}

We can now prove energy boundedness in
\begin{thm}
For $\eta\in (0,1)$, and conditions \eqref{eq:1DSigmaMatrixConditionsW2}--\eqref{eq:1DSigmaMatrixConditionsW3}
on the coupling matrices, the overset domain problem \eqref{eq:LS-RSSystem-MPu}-\eqref{eq:LS-RSSystem-MPv} is energy bounded with
$E(T) \le \inorm{\statevecGreek\omega_0}_\Omega^2$, and further, $\inorm{\statevec u}_{\Omega_u}^2 \le \inorm{\statevecGreek\omega_0}_\Omega^2/(1-\eta)$ and $\inorm{\statevec v}_{\Omega_v}^2 \le \inorm{\statevecGreek\omega_0}_\Omega^2/\eta$.
\label{thm:EnergyBound1DSystemPenalty}
\end{thm}
\begin{proof}
Defined in \eqref{eq:Pin1D}, the quantities $\mathcal P_b$ and $\mathcal P_c$ are non-negative under the coupling matrix conditions \eqref{eq:1DSigmaMatrixConditionsW2}--\eqref{eq:1DSigmaMatrixConditionsW3}. Then
\begin{equation}
\frac{d}{dt} E = \frac{d}{dt}\inorm{\statevec u}_{\Omega_u}^2 + \frac{d}{dt}\inorm{\statevec v}_{\Omega_v}^2  - \left\{ \eta\frac{d}{dt}\inorm{\statevec u}_{\Omega_{O}}^2 +
(1-\eta) \frac{d}{dt}\inorm{\statevec v}_{\Omega_{O}}^2 \right\}
\le0.
\end{equation}
Integrating in time, and using \eqref{eq:NewNorms1DW},
\begin{equation}
\begin{split}
E(T) =\inorm{\statevec u(T)}_{\Omega_u}^2 + &\inorm{\statevec v(T)}_{\Omega_v}^2  - \left\{ \eta \inorm{\statevec u(T)}_{\Omega_{O}}^2 +
(1-\eta)\inorm{\statevec v(T)}_{\Omega_{O}}^2 \right\}
\\&
\le \inorm{\statevecGreek \omega_0}_{\Omega_u}^2 +\inorm{\statevecGreek \omega_0}_{\Omega_v}^2  - \left\{ \eta\inorm{\statevecGreek \omega_0}_{\Omega_{O}}^2 +(1-\eta)
\inorm{\statevecGreek \omega_0}_{\Omega_{O}}^2 \right\}
\\&= \inorm{\statevecGreek \omega_0}^2_{\Omega} \ge 0.
\end{split}
\end{equation}
The last line follows because $\eta\inorm{\statevecGreek \omega_0}_{\Omega_{O}}^2 +(1-\eta)
\inorm{\statevecGreek \omega_0}_{\Omega_{O}}^2= \inorm{\statevecGreek \omega_0}_{\Omega_{O}}^2$, and $\inorm{\statevecGreek \omega_0}_{\Omega_u}^2 +\inorm{\statevecGreek \omega_0}_{\Omega_v}^2$ counts the overlap region twice.
Finally, by Lemma. \ref{lem:EBoundW}, $\inorm{\statevec u(T)}_{\Omega_u}^2 \le  \inorm{\omega_0}_\Omega^2/(1-\eta)$, $\inorm{\statevec v(T)}_{\Omega_v}^2\le  \inorm{\omega_0}_\Omega^2/\eta$ and the solutions at any time $T$ are bounded by the initial data.
\end{proof}
\subsubsection{Equivalence to the Original Problem and Well-Posedness}

The solutions to the overset problem \eqref{eq:LS-RSSystem-MPu}--\eqref{eq:LS-RSSystem-MPv}
are equivalent to the solution to the original problem, \eqref{eq:SystemForOmega1D}, from which it follows that the overset problem is well-posed. 
\begin{thm}
\label{thm:1DEquivalenceWithPenalty}
For $\eta\in (0,1)$, and conditions \eqref{eq:1DSigmaMatrixConditionsW2}--\eqref{eq:1DSigmaMatrixConditionsW3}
on the coupling matrices, the solutions of the overset domain problem \eqref{eq:LS-RSSystem-MPu}-\eqref{eq:LS-RSSystem-MPu} with interior penalties are equivalent to the solution for the full domain problem, i.e. $\statevec u = \statevecGreek\omega$ for $x\in\Omega_u$ and $\statevec v = \statevecGreek\omega$ for $x\in\Omega_v$. Under the assumption that the original problem is well-posed, the overset domain problem is also well-posed.
\end{thm}
\begin{proof}
Using the facts that the boundary conditions are linear, that the solution $\statevecGreek\omega$ satisfies the same PDE system in each subdomain as the subdomain solutions, and that $\statevec u - \statevec v = (\statevec u -\statevecGreek\omega) - (\statevec v - \statevecGreek\omega) = \statevec u' - \statevec v'$, the differences $\statevec u' = \statevec u -\statevecGreek\omega$ and  $\statevec v' = \statevec v -\statevecGreek\omega$ satisfy \eqref{eq:LS-RSSystem-MPu} and \eqref{eq:LS-RSSystem-MPv}, respectively, with trivial initial conditions, $\statevec u'(x,0) = \statevecGreek\omega'_0 \equiv 0$ and  $\statevec v'(x,0) = \statevecGreek\omega'_0 \equiv 0$. Then by Thm. \ref{thm:EnergyBound1DSystemPenalty}, for any time $t=T$,
\begin{equation}
\begin{gathered}
    \inorm{\statevec u'(T)}_{\Omega_u}^2 \le  \inorm{\statevecGreek\omega'_0}_\Omega^2/(1-\eta)=0\hfill\\
    \inorm{\statevec v'(T)}_{\Omega_v}^2 \le  \inorm{\statevecGreek\omega'_0}_\Omega^2/\eta=0,\hfill
\end{gathered}
\end{equation}
from which the equivalence follows. Since the original problem for $\statevecGreek\omega$ is well posed, it follows that the overset problem is also well-posed.
\end{proof}

\subsubsection{Conservation}
The overset problem \eqref{eq:LS-RSSystem-MPu}--\eqref{eq:LS-RSSystem-MPv} with the interior penalties is also conservative, and the flux lost from one domain at the interface is equal to that gained by the other, if the overset domain problem is energy bounded. 
To show conservation, we set $\statevecGreek\phi_u = \statevecGreek\phi_v = \statevecGreek\phi$ and combine \eqref{eq:WeakFormsUV} and \eqref{eq:WeakFormsOverlap} as in \eqref{eq:EnergyTimeDerivW},
\begin{equation}
\begin{gathered}
 \iprod{\statevecGreek\phi,\statevec u_t}_{\Omega_u} 
 + \left.\statevecGreek\phi^T\mmatrix A \statevec u \right|_a^c 
 - \iprod{\statevecGreek\phi',\mmatrix A \statevec u }_{\Omega_u}
 +\left.\statevecGreek\phi^T\ \mmatrix{$\Sigma$}_u^b (\statevec u-\statevec v)\right|_b+\left.\statevecGreek\phi^T\ \mmatrix{$\Sigma$}_u^c (\statevec u-\statevec v)\right|_c\hfill\\
+ \iprod{\statevecGreek\phi,\statevec v_t}_{\Omega_v} +  \left.\statevecGreek\phi^T\mmatrix A \statevec v \right|_b^d 
- \iprod{\statevecGreek\phi',\mmatrix A \statevec v}_{\Omega_v} + \left.\statevecGreek\phi^T \mmatrix{$\Sigma$}_v^b (\statevec v-\statevec u)\right|_b+  \left.\statevecGreek\phi^T \mmatrix{$\Sigma$}_v^c (\statevec v-\statevec u)\right|_c\hfill\\
-\eta\left\{ \iprod{\statevecGreek\phi,\statevec u_t}_{\Omega_O} + \left.\statevecGreek\phi^T\mmatrix A \statevec u \right|_b^c - \iprod{\statevecGreek\phi',\mmatrix A \statevec u}_{\Omega_O}\right\}\hfill\\
-(1-\eta)\left\{\iprod{\statevecGreek\phi,\statevec v_t}_{\Omega_O} + \left.\statevecGreek\phi^T\mmatrix A \statevec v \right|_b^c - \iprod{\statevecGreek\phi',\mmatrix A \statevec v}_{\Omega_O}\right\} = 0.\hfill\\
\end{gathered}
\end{equation}
Gathering terms, 
\begin{equation}
\begin{split}
&\iprod{\statevecGreek\phi,\statevec u_t}_{\Omega_u} + \iprod{\statevecGreek\phi,\statevec v_t}_{\Omega_v} - \left\{  \eta\iprod{\statevecGreek\phi,\statevec u_t}_{\Omega_O} + (1-\eta)\iprod{\statevecGreek\phi,\statevec v_t}_{\Omega_O} \right\}+ \statevecGreek\phi^T\left(\left.\mmatrix A\statevec v\right|_d - \left.\mmatrix A\statevec u\right|_a\right)
\\&- 
\left\{ \iprod{\statevecGreek\phi',\mmatrix A \statevec u }_{\Omega_u} + \iprod{\statevecGreek\phi',\mmatrix A \statevec v}_{\Omega_v} -
\eta \iprod{\statevecGreek\phi',\mmatrix A \statevec u}_{\Omega_O}- (1-\eta)  \iprod{\statevecGreek\phi',\mmatrix A \statevec v}_{\Omega_O}  \right\}
\\& +
\left.\statevecGreek\phi^T\left\{ \eta \mmatrix A + \mmatrix{$\Sigma$}^b_L - \mmatrix{$\Sigma$}^b_R\right\}(\statevec u - \statevec v)\right|_b
\\&
+\left.\statevecGreek\phi^T \left\{ (1-\eta) \mmatrix A + \mmatrix{$\Sigma$}^c_L - \mmatrix{$\Sigma$}^c_R\right\}(\statevec u - \statevec v)\right|_c.
\end{split}
\label{eq:General1DConservationFormula}
\end{equation}
Making the interior boundary term vanish, we have the following theorem:
\begin{thm}
The overset domain problem \eqref{eq:LS-RSSystem-MPu} is conservative if the interface terms at $x=b$ and $x=c$ in \eqref{eq:General1DConservationFormula} vanish, that is, if
\begin{equation}
 \beta \mmatrix A + \mmatrix{$\Sigma$}_u =  \mmatrix{$\Sigma$}_v,
\label{eq:ConservationCondition}
\end{equation}
where, as before, $\beta = \eta$ at $x=b$ and $\beta = 1-\eta$ at $x=c$.
\end{thm}
With condition \eqref{eq:ConservationCondition}, the flux lost from one penalty on one domain is gained at the same point in the other domain, leading to overall conservation. This conservation condition is the first of the energy boundedness conditions, \eqref{eq:1DSigmaMatrixConditionsW2}.

Most importantly, the meaning of conservation in the overset context can be interpreted in terms of the original problem over $\Omega$: 
\begin{thm}
Conservation in the overset domain problem \eqref{eq:LS-RSSystem-MPu}-\eqref{eq:LS-RSSystem-MPu} with the condition \eqref{eq:ConservationCondition} is equivalent to
\begin{equation}
    \frac{d}{dt}\int_a^d\statevecGreek\omega dx  = \statevec f_a-\statevec f_d,
\end{equation}
where $\statevec f_a = \mmatrix A\statevec u|_a$ and $\statevec f_d = \mmatrix A\statevec v|_d$ are the boundary fluxes at $x=a$ and $x=d$, repsectively.
\label{thm:1DPenaltyConservationMatch}
\end{thm}
\begin{proof}
Summing the integrals over $\Omega_u$ and $\Omega_v$ in \eqref{eq:General1DConservationFormula} counts the overlap region twice, so the extra contribution is removed. With the conservation condition \eqref{eq:ConservationCondition} and $\phi = 1$ in \eqref{eq:General1DConservationFormula},
\begin{equation}
\frac{d}{dt}\left\{\int_a^c\statevec u dx + \int_b^d\statevec v dx - \left( \eta \int_b^c\statevec udx +  (1-\eta)\int_b^c\statevec vdx\right)\right\}= \left(\mmatrix A\statevec u|_a - \mmatrix A\statevec v|_d\right) = \statevec f_a-\statevec f_d.
\label{eq:Conservation1D}
\end{equation}
By Thm. \ref{thm:1DEquivalenceWithPenalty}, $\statevec u = \statevecGreek \omega$ and $\statevec v = \statevecGreek \omega$ on their respective domains, the integrals collapse to the single integral over $[a,d]$, and the result follows.
\end{proof}

\section{Well-Posed Overset Domain Problems in Two Space Dimensions}

 The procedure of adding penalty terms at the interfaces to get a well-posed overset domain problem extends to more than one space dimension. In this section we consider two types of overset problems: one with coupling only at the domain interfaces, like that done in Sec. \ref{sec:WeightedMultiplePenalty},  and another that also adds coupling within the interior of the overlap.

We define the full domain problem on the domain $\Omega$ shown in Fig. \ref{fig:2DGemetry} as
\begin{equation}
\left\{\begin{gathered}
\statevecGreek \omega_t + \spacevec{\mmatrix A}\cdot\nabla\statevecGreek \omega = 0,\quad \spacevec x\in\Omega\hfill\\
\statevecGreek \omega(\spacevec x,t) = B^a\left( \statevecGreek \omega(\spacevec x,t)\right)\quad \spacevec x \in \Gamma_a\hfill\\
\statevecGreek \omega(\spacevec x,t) = B^d\left( \statevecGreek \omega(\spacevec x,t)\right)\quad \spacevec x \in \Gamma_d\hfill\\
\statevecGreek \omega(\spacevec x,0) = \statevecGreek \omega_0(\spacevec x),\quad \spacevec x\in\Omega, \hfill
\end{gathered} \right.
\label{eq:Omega2DProblem}
\end{equation}
where
$\spacevec{\mmatrix A }= \mmatrix A_1\hat x + \mmatrix A_2\hat y$. As in one space dimension, we require that the physical boundary operators $B^a$ and $B^d$ be characteristic, dissipative, and impose trivial conditions on the incoming characteristic variables.
% so that $\statevecGreek \omega^T B^a(\statevecGreek \omega)\ge 0$ and $\statevecGreek \omega^T B^d(\statevecGreek \omega)\ge 0$. They will have the same form as \eqref{eq:System1DBCs}.

\begin{figure}[tbp] %  figure placement: here, top, bottom, or page
   \centering
   \includegraphics[width=2.5in]{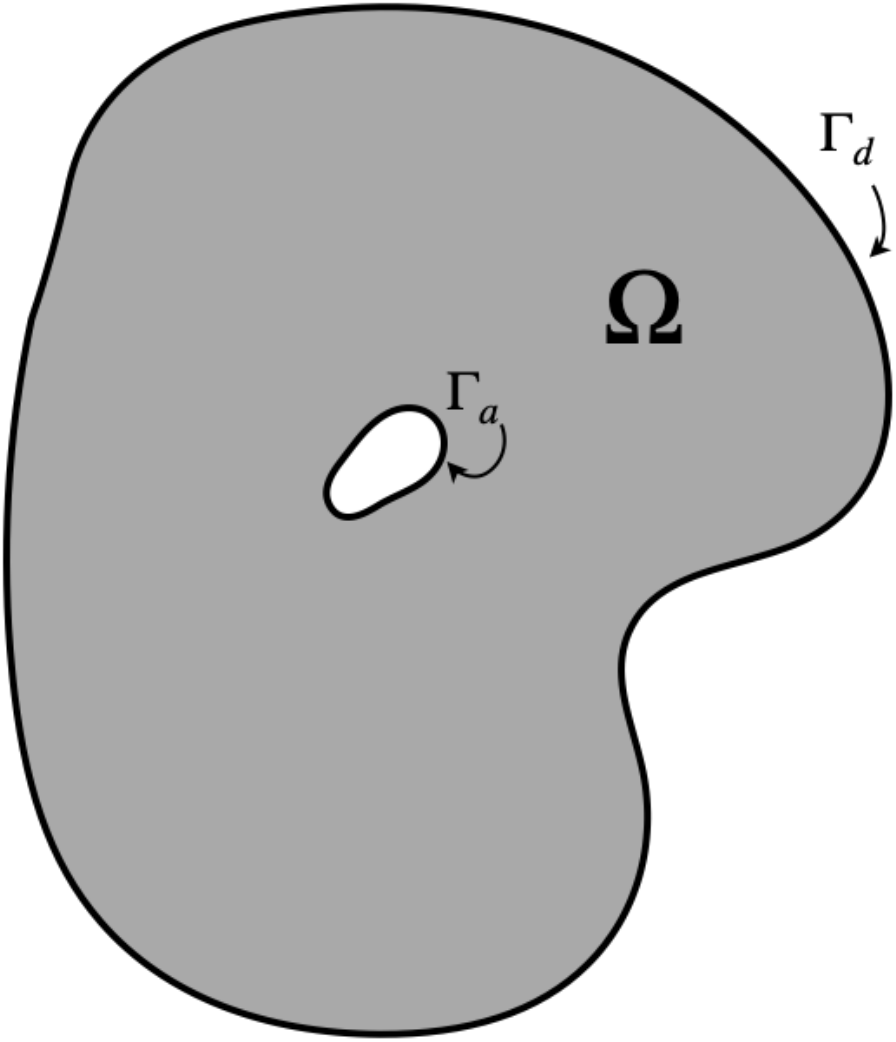} 
   \caption{Diagram of the 2D geometry}
   \label{fig:2DGemetry}
\end{figure}

The overset geometry that covers the domain $\Omega$ in Fig. \ref{fig:2DGemetry} is shown in Fig. \ref{fig:2DOversetGemetry}. The base domain, $\Omega_v$, has a hole bounded by the curve $\Gamma_b$. The overset domain extends beyond that hole to a boundary bounded by the curve, $\Gamma_c$. In this way, the physical and interface boundaries are ordered as they were in one space dimension.
\begin{figure}[htbp] %  figure placement: here, top, bottom, or page
   \centering
   \includegraphics[width=5.5in]{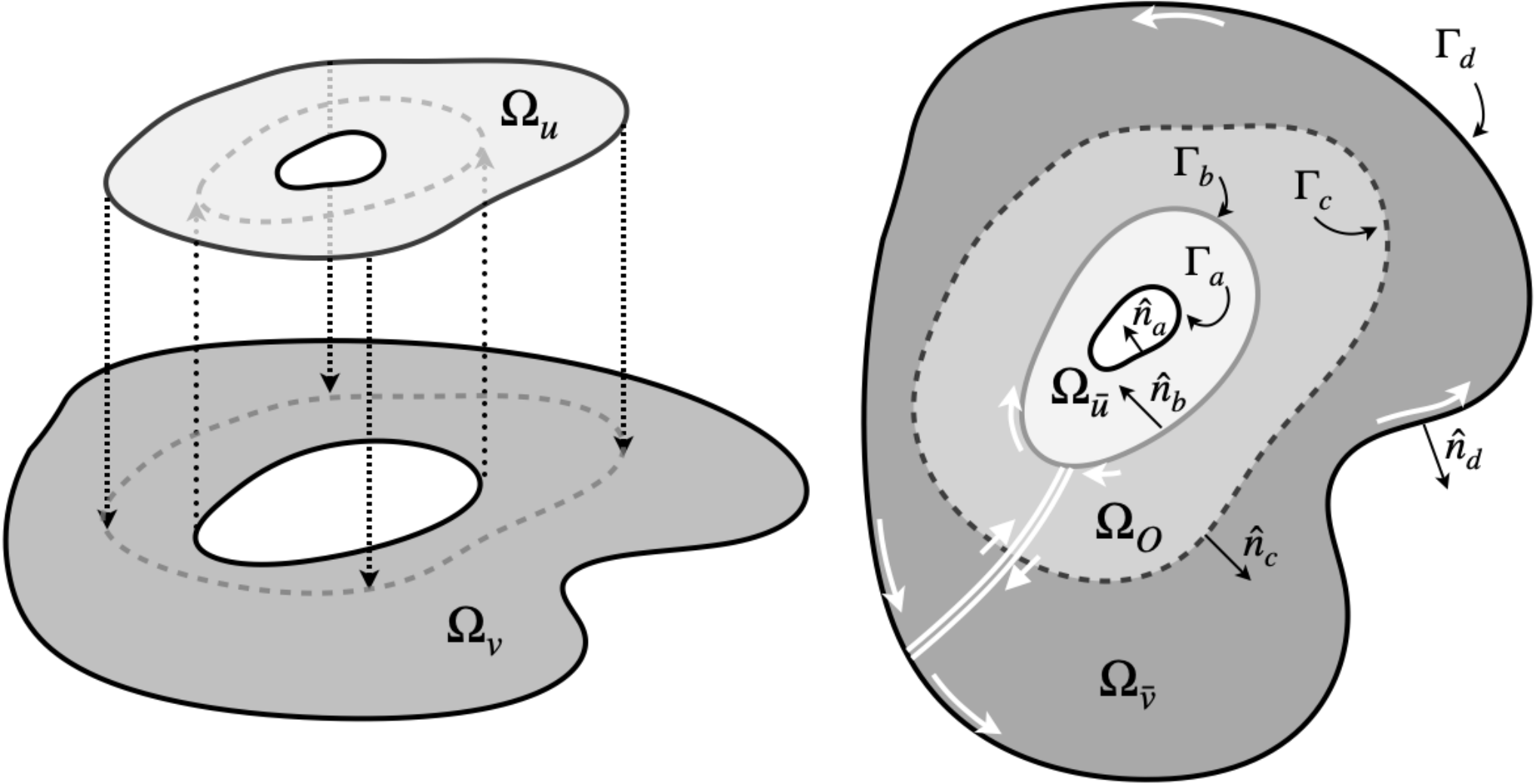} 
   \caption{Diagrams of the overset geometry in 2D}
   \label{fig:2DOversetGemetry}
\end{figure}

\subsection{Boundary only Coupling}
The two-dimensional versions of \eqref{eq:LS-RSSystem-MPu} and \eqref{eq:LS-RSSystem-MPv} on the domains $\Omega_u$ and $\Omega_v$ are the overset system problem, $O$,
\begin{equation}
O\;\left\{\begin{gathered}
\statevec u_t + \spacevec{\mmatrix A}\cdot\nabla\statevec u  + \mathcal L_u^b \left[ \mmatrix{$\Sigma$}_u^b (\statevec u-\statevec v)\right] 
+ \mathcal L_u^c \left[ \mmatrix{$\Sigma$}_u^c (\statevec u-\statevec v)\right]= 0, \quad\spacevec x\in\Omega_u\hfill\\
\statevec u(\spacevec x,t) = B^a\left( \statevec u(\spacevec x,t)\right)\quad \spacevec x \in \Gamma_a\hfill\\
\statevec u(\spacevec x,0) = \statevecGreek \omega_0(\spacevec x), \quad\spacevec x\in\Omega_u \hfill
\end{gathered} \right.
\label{eq:OmegaU-RSSystem-MPu}
\end{equation}
and the base system problem, $B$,
\begin{equation}
B\;\left\{\begin{gathered}
\statevec v_t + \spacevec{\mmatrix A}\cdot\nabla\statevec v + \mathcal L_v^b \left[ \mmatrix{$\Sigma$}_v^b (\statevec v-\statevec u)\right]
+ \mathcal L_v^c \left[ \mmatrix{$\Sigma$}_v^c (\statevec v-\statevec u)\right]= 0, \quad\spacevec x\in\Omega_v\hfill\\
\statevec v(\spacevec x,t) = B^d\left( \statevec v(\spacevec x,t)\right)\quad \spacevec x \in \Gamma_d\hfill\\
\statevec v(\spacevec x,0) = \statevecGreek \omega_0(\spacevec x ),\quad \spacevec x\in\Omega_v .\hfill
\end{gathered} \right.
\label{eq:OmegaV-RSSystem-MPv}
\end{equation}
 
The functions $\mathcal L_{u}^b$, etc. are lifting operators that return the values along a curve. Specifically, over a domain $V$ containing the curve $\Gamma$,
\begin{equation}
\int_V \statevecGreek \psi^T\mathcal L[\statevec f]dV = \int_\Gamma \statevecGreek \psi^T(\spacevec x(s))\statevec f(\spacevec x(s))ds,
\end{equation}
where $\statevecGreek \psi$ is again some test function.

As in one space dimension, we create weak forms over $\Omega_u$ and $\Omega_v$,
\begin{equation}
\iprod{\statevecGreek\phi_u,\statevec u_t}_{\Omega_u} 
+ \iprod{\statevecGreek\phi_u, \spacevec{\mmatrix A}\cdot\nabla\statevec u}_{\Omega_u}   
+  \oint_{\Gamma_b} \statevecGreek\phi^T_u\mmatrix{$\Sigma$}_u^b (\statevec u-\statevec v)ds 
+ \oint_{\Gamma_c}\statevecGreek\phi^T_u\mmatrix{$\Sigma$}_u^c (\statevec u-\statevec v)ds= 0,
\label{eq:OmegaU-RSSystem-MPuW}
\end{equation}
\begin{equation}
\iprod{\statevecGreek\phi_v,\statevec v_t}_{\Omega_v} 
+ \iprod{\statevecGreek\phi_v, \spacevec{\mmatrix A}\cdot\nabla\statevec v }_{\Omega_v}
+ \oint_{\Gamma_b} \statevecGreek\phi^T_v\mmatrix{$\Sigma$}_v^b (\statevec v-\statevec u)ds
+ \oint_{\Gamma_c} \statevecGreek\phi^T_v\ \mmatrix{$\Sigma$}_v^c (\statevec v-\statevec u)ds= 0,
\label{eq:OmegaV-RSSystem-MPvW}
\end{equation}
where for some $\statevecGreek \phi$ and $\statevec q$,
\begin{equation}
\iprod{\statevecGreek \phi,\statevec q}_\Omega = \int_\Omega \statevecGreek \phi^T\statevec q dxdy.
\end{equation}
Similarly, in the overlap region, $\Omega_O$,
\begin{equation}
\begin{gathered}
\iprod{\statevecGreek\phi_u,\statevec u_t}_{\Omega_O} + \iprod{\statevecGreek\phi_u, \spacevec{\mmatrix A}\cdot\nabla\statevec u}_{\Omega_O}  = 0,
\hfill\\
\iprod{\statevecGreek\phi_v,\statevec v_t}_{\Omega_O} + \iprod{\statevecGreek\phi_v, \spacevec{\mmatrix A}\cdot\nabla\statevec v }_{\Omega_O} = 0.
\hfill
\end{gathered}
\label{eq:2DWeakFormOveralap}
\end{equation}

\subsubsection{The Energy Bound}

To show energy boundedness of the problems \eqref{eq:OmegaU-RSSystem-MPu}--\eqref{eq:OmegaV-RSSystem-MPv}, we choose $\statevecGreek \phi_u = \statevec u$ and $\statevecGreek \phi_v = \statevec v$. Multidimensional integration by parts applied to the divergence terms and the fact that the coefficient matrices are symmetric and constant says that 
\begin{equation}
\iprod{\statevec u,\spacevec{\mmatrix A}\cdot\nabla\statevec u}_{\Omega_u} = 
\oneHalf \oint_{\Gamma_a} \statevec u^T{\spacevec{\mmatrix A}\cdot\hat n_a}\statevec u ds
+ \oneHalf \oint_{\Gamma_c} \statevec u^T{\spacevec{\mmatrix A}\cdot\hat n_c}\statevec u ds,
\label{eq:u2DIntByParts}
\end{equation}
where the integrals along cuts sketched in Fig. \ref{fig:2DOversetGemetry} cancel. 
Similarly,
\begin{equation}
\iprod{\statevec v,\spacevec{\mmatrix A}\cdot\nabla\statevec v}_{\Omega_v} = \oneHalf \int_{\Gamma_b} \statevec v^T\spacevec{\mmatrix A}\cdot\hat n_b\statevec v ds + \oneHalf \int_{\Gamma_d} \statevec v^T{\spacevec{\mmatrix A}\cdot\hat n_d}\statevec v ds.
\label{eq:v2DIntByParts}
\end{equation}

After substituting \eqref{eq:u2DIntByParts} and \eqref{eq:v2DIntByParts} into \eqref{eq:OmegaU-RSSystem-MPuW} and \eqref{eq:OmegaV-RSSystem-MPvW}, gathering like integrals,  and multiplying by two,
\begin{equation}
 \frac{d}{dt} \inorm{\statevec u}^2_{\Omega_u} 
 +  \oint_{\Gamma_a} \statevec u^T{\spacevec{\mmatrix A}\cdot\hat n_a}\statevec u ds
+  \oint_{\Gamma_b}2\statevec u^T\mmatrix{$\Sigma$}_u^b (\statevec u-\statevec v)ds
+ \oint_{\Gamma_c} \left\{\statevec u^T\spacevec{\mmatrix A}\cdot\hat n_c\statevec u  + 2\statevec u^T\mmatrix{$\Sigma$}_u^c (\statevec u-\statevec v)\right\}ds   
= 0
\label{eq:U2DNorm}
\end{equation}
and
\begin{equation}
 \frac{d}{dt} \inorm{\statevec v}^2_{\Omega_v}  
+ \oint_{\Gamma_b}\left\{\statevec v^T\spacevec{\mmatrix A}\cdot\hat n_b\statevec v +  2\statevec v^T \mmatrix{$\Sigma$}_v^b (\statevec v-\statevec u)\right\}ds
+ \oint_{\Gamma_c} 2\statevec v^T\mmatrix{$\Sigma$}_v^c (\statevec v-\statevec u)ds
+  \oint_{\Gamma_d} \statevec v^T\spacevec{\mmatrix A}\cdot\hat n_d\statevec v ds
= 0,
\end{equation}
whereas the overlap region inner products become
\begin{equation}
\begin{gathered}
 \frac{d}{dt} \inorm{\statevec u}^2_{\Omega_O} 
+  \oint_{\Gamma_b} \statevec u^T{\spacevec{\mmatrix A}\cdot\hat n_b}\statevec u ds 
+  \oint_{\Gamma_c} \statevec u^T{\spacevec{\mmatrix A}\cdot\hat n_c}\statevec u ds  = 0
\hfill\\
 \frac{d}{dt} \inorm{\statevec v}^2_{\Omega_O} 
+  \oint_{\Gamma_b} \statevec v^T{\spacevec{\mmatrix A}\cdot\hat n_b}\statevec v ds 
+  \oint_{\Gamma_c} \statevec v^T{\spacevec{\mmatrix A}\cdot\hat n_c}\statevec v ds
 = 0.
\hfill
\end{gathered}
\label{eq:Overlap2Dnorms}
\end{equation}

Before moving further, we apply the boundary conditions along the inner physical boundary, $\Gamma_a$, and the outer physical boundary, $\Gamma_d$. For each, we split the matrix $\spacevec{\mmatrix A}\cdot\hat n = \mmatrix A^+ - |\mmatrix A^-|$  according to the positive and negative eigenvalues
% so that 
%\begin{equation}
%\begin{split}
%\oint_{\Gamma_c} \statevec v^T\spacevec{\mmatrix A}\cdot\hat n_c\statevec v ds 
%&= \oint_{\Gamma_c} \statevec v^T{\mmatrix A^+}\statevec v ds - \oint_{\Gamma_c} \statevec v^T{|\mmatrix A^-|}\statevec v ds
%\\&
%= \oint_{\Gamma_c} \statevec v^T{\mmatrix A^+}\statevec v ds - \oint_{\Gamma_c} \statevec w^{-,T}{|\bar{\mmatrix A}^-|}\statevec w^- ds.
%\end{split}
%\end{equation}
%Applying 
and set the boundary condition $\statevec w^- = 0$, so that the incoming characteristic data, relative to the outward normal, is zero. Then
\begin{equation}
\begin{gathered}
\oint_{\Gamma_a} \statevec u^T\spacevec{\mmatrix A}\cdot\hat n_a\statevec u ds 
= \oint_{\Gamma_a} \statevec u^T{\mmatrix A^+}\statevec u ds \ge 0, \\ \hfill
\oint_{\Gamma_d} \statevec v^T\spacevec{\mmatrix A}\cdot\hat n_d\statevec v ds 
= \oint_{\Gamma_d} \statevec v^T{\mmatrix A^+}\statevec v ds \ge 0.
\end{gathered}
\label{eq:2DBCDissipation}
\end{equation}

We then construct the time derivative of the energy norm \eqref{eq:EnergyNorm} from \eqref{eq:U2DNorm}--\eqref{eq:Overlap2Dnorms}, so that with the boundary conditions \eqref{eq:2DBCDissipation} applied,
\begin{equation}
\begin{split}
\frac{d}{dt} E &= \frac{d}{dt}\inorm{\statevec u}_{\Omega_u}^2 + \frac{d}{dt}\inorm{\statevec v}_{\Omega_v}^2  - \left\{ \eta \frac{d}{dt}\inorm{\statevec u}_{\Omega_{O}}^2 +
 (1-\eta)\frac{d}{dt}\inorm{\statevec v}_{\Omega_{O}}^2 \right\}
 \\&
 + \oint_{\Gamma_b} \left\{2\statevec u^T\mmatrix{$\Sigma$}_u^b (\statevec u-\statevec v) 
 +\statevec v^T\spacevec{\mmatrix A}\cdot\hat n_b\statevec v 
 +  2\statevec v^T \mmatrix{$\Sigma$}_v^b (\statevec v-\statevec u) 
 -\eta  \statevec u^T{\spacevec{\mmatrix A}\cdot\hat n_b}\statevec u 
 -(1-\eta)\statevec v^T{\spacevec{\mmatrix A}\cdot\hat n_b}\statevec v\right\}              ds
 \\&
 + \oint_{\Gamma_c} \left\{\statevec u^T\spacevec{\mmatrix A}\cdot\hat n_c\statevec u  
 + 2\statevec u^T\mmatrix{$\Sigma$}_u^c (\statevec u-\statevec v) + 2\statevec v^T\mmatrix{$\Sigma$}_v^c (\statevec v-\statevec u)
 -\eta \statevec u^T{\spacevec{\mmatrix A}\cdot\hat n_c}\statevec u 
 -(1-\eta)\statevec v^T{\spacevec{\mmatrix A}\cdot\hat n_c}\statevec v\right\}
 ds
 \\& \le 0,
 \end{split}
 \label{eq:2DEnergyTimeDerivative}
\end{equation}
for $0 < \eta < 1$.

Energy boundedness therefore depends on the signs of the integrands of the line integrals. Gathering the terms in the integrands, let us write
\begin{equation}
\mathcal P_b =  -\eta\statevec u^T{\spacevec{\mmatrix A}\cdot\hat n_b}\statevec u + \eta\statevec v^T{\spacevec{\mmatrix A}\cdot\hat n_b}\statevec v +2\statevec u^T\mmatrix{$\Sigma$}_u^b (\statevec u-\statevec v) +2\statevec v^T \mmatrix{$\Sigma$}_v^b (\statevec v-\statevec u)
\end{equation}
and
\begin{equation}
\mathcal P_c = (1-\eta)\statevec u^T\spacevec{\mmatrix A}\cdot\hat n_c\statevec u -(1-\eta)\statevec v^T{\spacevec{\mmatrix A}\cdot\hat n_c}\statevec v + 2\statevec u^T\mmatrix{$\Sigma$}_u^c (\statevec u-\statevec v) + 2\statevec v^T\mmatrix{$\Sigma$}_v^c (\statevec v-\statevec u).
\end{equation}
As before, \eqref{eq:2DEnergyTimeDerivative} implies energy boundedness if $\mathcal P_a\ge0$ and $\mathcal P_b\ge 0$. 

Both  $\mathcal P_b$ and  $\mathcal P_c$ are of the same form, namely
\begin{equation}
\mathcal P = \beta\statevec u^T{\mmatrix A}\statevec u -\beta\statevec v^T{{\mmatrix A}}\statevec v + 2\statevec u^T\mmatrix{$\Sigma$}_u (\statevec u-\statevec v) + 2\statevec v^T\mmatrix{$\Sigma$}_v (\statevec v-\statevec u),
\label{eq:Pin2D}
\end{equation}
where 
\begin{equation}
\beta = \left\{
\begin{gathered}
-\eta, \quad \text{on }\Gamma_b \\ \hfill
(1-\eta),\text{on } \Gamma_c
\end{gathered}
   \right.
\label{eq:BetaCondition2D}
\end{equation}
and $\mmatrix A = \spacevec{\mmatrix A}\cdot\hat n$. Both are in the same form as the equivalents in one space dimension, \eqref{eq:Pin1D}.
So, as before, we re-write $\mathcal P$ as
\begin{equation}
\mathcal P =\left[\begin{array}{c}\statevec u \\\statevec v\end{array}\right]^T\mmatrix M  \left[\begin{array}{c}\statevec u \\\statevec v\end{array}\right] ,
\end{equation}
where
\begin{equation}
\mmatrix M =\left[\begin{array}{cc}\beta A + 2\mmatrix{$\Sigma$}_u& -(\mmatrix{$\Sigma$}_u  + \mmatrix{$\Sigma$}_v)\\-(\mmatrix{$\Sigma$}_u+ \mmatrix{$\Sigma$}_v) & -\beta\mmatrix A + 2\mmatrix{$\Sigma$}_v \end{array}\right].
\end{equation}
Therefore, a sufficient condition for energy boundedness is that $\mmatrix M\ge 0$, which we showed in Sec. \ref{sec:WeightedMultiplePenalty}  occurs when the conditions in \eqref{eq:1DSigmaMatrixConditionsW} are statisfied.

Under the conditions, \eqref{eq:1DSigmaMatrixConditionsW}, $\mathcal P_a, \mathcal P_b\ge 0$, and
\begin{equation}
\frac{d}{dt} E = \frac{d}{dt}\inorm{\statevec u}_{\Omega_u}^2 + \frac{d}{dt}\inorm{\statevec v}_{\Omega_v}^2  - \left\{ \eta\frac{d}{dt}\inorm{\statevec u}_{\Omega_{O}}^2 +
(1-\eta) \frac{d}{dt}\inorm{\statevec v}_{\Omega_{O}}^2 \right\}
\le0.
\label{eq:2DEnergyEquation}
\end{equation}
Integrating \eqref{eq:2DEnergyEquation} in time, and using \eqref{eq:NewNorms1DW},
\begin{equation}
\begin{split}
E(T) = \inorm{\statevec u(T)}_{\Omega_u}^2 + &\inorm{\statevec v(T)}_{\Omega_v}^2  - \left\{ \eta \inorm{\statevec u(T)}_{\Omega_{O}}^2 +
(1-\eta)\inorm{\statevec v(T)}_{\Omega_{O}}^2 \right\}
\\&
\le \inorm{\statevecGreek\omega_0}_{\Omega_u}^2 +\inorm{\statevecGreek\omega_0}_{\Omega_v}^2  - \left\{ \eta\inorm{\statevecGreek\omega_0}_{\Omega_{O}}^2 +(1-\eta)
\inorm{\statevecGreek\omega_0}_{\Omega_{O}}^2 \right\}
\\&= \inorm{\statevecGreek\omega_0}^2_{\Omega}.
\end{split}
\end{equation}
Then by Thm. \ref{lem:EBoundW}, $\inorm{\statevec u(T)}_{\Omega_u}^2 \le \inorm{\statevecGreek\omega_0}_{\Omega_{O}}^2/(1-\eta)$, $\inorm{\statevec v(T)}_{\Omega_v}^2\le \inorm{\statevecGreek\omega_0}_{\Omega_{O}}^2/\eta$ and the solutions at any time $T$ are bounded by the data.
We have therefore proved
\begin{thm}
For $\eta\in (0,1)$, and conditions \eqref{eq:1DSigmaMatrixConditionsW2}--\eqref{eq:1DSigmaMatrixConditionsW3}
on the coupling matrices, where $\beta$ is given by \eqref{eq:BetaCondition2D}, the overset domain problem \eqref{eq:OmegaU-RSSystem-MPu}-\eqref{eq:OmegaV-RSSystem-MPv} is energy bounded with
$E(T) \le \inorm{\statevecGreek \omega_0}_\Omega^2$, and further, $\inorm{\statevec u}_{\Omega_u}^2 \le \inorm{\statevecGreek\omega_0}_\Omega^2/(1-\eta)$ and $\inorm{\statevec v}_{\Omega_v}^2 \le \inorm{\statevecGreek\omega_0}_\Omega^2/\eta$.
\end{thm}
Furthermore, we have
\begin{thm}
For $\eta\in (0,1)$, and conditions \eqref{eq:1DSigmaMatrixConditionsW2}--\eqref{eq:1DSigmaMatrixConditionsW3}
on the coupling matrices, the solutions of the overset domain problem \eqref{eq:OmegaU-RSSystem-MPu}-\eqref{eq:OmegaV-RSSystem-MPv} are equivalent to the solution for the full domain problem, i.e. $\statevec u = \statevecGreek\omega$ for $\spacevec x\in\Omega_u$ and $\statevec v = \statevecGreek\omega$ for $\spacevec x\in\Omega_v$. Under the assumption that the original problem is well-posed, the overset domain problem is well-posed.
\label{thm:2DEquivalence}
\end{thm}
\begin{proof}
The proof follows that of Thm. \ref{thm:1DEquivalenceWithPenalty}.
\end{proof}

\subsubsection{Conservation}
Again, to show conservation we set $\statevecGreek\phi_u = \statevecGreek\phi_v = \statevecGreek\phi$, now in \eqref{eq:OmegaU-RSSystem-MPuW}--\eqref{eq:OmegaV-RSSystem-MPvW} and \eqref{eq:2DWeakFormOveralap}, apply integration by parts to the divergence terms, and combine as in constructing the energy to get
\begin{equation}
\begin{gathered}
 \iprod{\statevecGreek\phi,\statevec u_t}_{\Omega_u} 
 + \oint_{\Gamma_a}\statevecGreek\phi^T\spacevec{\mmatrix A}\cdot\hat n_a \statevec u ds
  + \oint_{\Gamma_c}\statevecGreek\phi^T\spacevec{\mmatrix A}\cdot\hat n_c \statevec u ds
 - \iprod{\spacevec{\mmatrix A}\cdot\nabla\statevecGreek\phi,\statevec u }_{\Omega_u}
 +\oint_{\Gamma_b}\statevecGreek\phi^T\mmatrix{$\Sigma$}_u^b (\statevec u-\statevec v)ds
 +\oint_{\Gamma_c}\statevecGreek\phi^T\mmatrix{$\Sigma$}_u^c (\statevec u-\statevec v)ds\hfill\\
+ \iprod{\statevecGreek\phi,\statevec v_t}_{\Omega_v} 
+  \oint_{\Gamma_b}\statevecGreek\phi^T\spacevec{\mmatrix A}\cdot\hat n_b  \statevec v ds 
+  \oint_{\Gamma_d}\statevecGreek\phi^T\spacevec{\mmatrix A}\cdot\hat n_d  \statevec v ds
- \iprod{\spacevec{\mmatrix A}\cdot\nabla\statevecGreek\phi, \statevec v}_{\Omega_v} 
+ \oint_{\Gamma_b}\statevecGreek\phi^T\mmatrix{$\Sigma$}_v^b (\statevec v-\statevec u)ds
+  \oint_{\Gamma_c}\statevecGreek\phi^T\mmatrix{$\Sigma$}_v^c (\statevec v-\statevec u)ds\hfill\\
-\eta\left\{ \iprod{\statevecGreek\phi,\statevec u_t}_{\Omega_O} 
+ \oint_{\Gamma_b}\statevecGreek\phi^T\spacevec{\mmatrix A}\cdot\hat n_b\statevec u ds
+ \oint_{\Gamma_c}\statevecGreek\phi^T\spacevec{\mmatrix A}\cdot\hat n_c\statevec u ds
- \iprod{\spacevec{\mmatrix A}\cdot\nabla \statevecGreek\phi, \statevec u}_{\Omega_O}
\right\}\hfill\\
-(1-\eta)\left\{\iprod{\statevecGreek\phi,\statevec v_t}_{\Omega_O} 
+  \oint_{\Gamma_b}\statevecGreek\phi^T\spacevec{\mmatrix A}\cdot\hat n_b \statevec v ds
+  \oint_{\Gamma_c}\statevecGreek\phi^T\spacevec{\mmatrix A}\cdot\hat n_c \statevec v ds
- \iprod{\spacevec{\mmatrix A}\cdot\nabla \statevecGreek\phi,\statevec v}_{\Omega_O}
\right\} \hfill\\
= 0.\hfill\\
\end{gathered}
\end{equation}
Following the procedure used for one space dimension, we then gather area, surface and boundary terms and re-organize
\begin{equation}
\begin{gathered}
 \iprod{\statevecGreek\phi,\statevec u_t}_{\Omega_u}  
 +  \iprod{\statevecGreek\phi,\statevec v_t}_{\Omega_v}  
 -\eta\iprod{\statevecGreek\phi,\statevec u_t}_{\Omega_O} 
 -(1-\eta)\iprod{\statevecGreek\phi,\statevec v_t}_{\Omega_O} \hfill\\
 -\left\{   
 \iprod{\spacevec{\mmatrix A}\cdot\nabla\statevecGreek\phi,\statevec u }_{\Omega_u} 
 + \iprod{\spacevec{\mmatrix A}\cdot\nabla\statevecGreek\phi, \statevec v}_{\Omega_v} 
 -\eta\iprod{\spacevec{\mmatrix A}\cdot\nabla \statevecGreek\phi, \statevec u}_{\Omega_O} 
 - (1-\eta) \iprod{\spacevec{\mmatrix A}\cdot\nabla \statevecGreek\phi,\statevec v}_{\Omega_O}
 \right\} \hfill\\
 +  \oint_{\Gamma_b} \statevecGreek\phi^T\left\{
\left(\mmatrix{$\Sigma$}_u^b -\mmatrix{$\Sigma$}_v^b\right)  
 -\eta\spacevec{\mmatrix A}\cdot\hat n_b
 \right\}(\statevec u-\statevec v)ds \hfill\\
  +  \oint_{\Gamma_c}
  \statevecGreek\phi^T\left\{
  \left(\mmatrix{$\Sigma$}_u^c -\mmatrix{$\Sigma$}_v^c\right)
  +(1-\eta)\spacevec{\mmatrix A}\cdot\hat n_c 
 \right\}(\statevec u - \statevec v)ds \hfill\\
+ \oint_{\Gamma_a}\statevecGreek\phi^T\spacevec{\mmatrix A}\cdot\hat n_a \statevec u ds
+  \oint_{\Gamma_d}\statevecGreek\phi^T\spacevec{\mmatrix A}\cdot\hat n_d  \statevec v ds \hfill\\
= 0.\hfill
\end{gathered}
\label{eq:2DConservationEquation}
\end{equation}
As in one space dimension, we have
\begin{thm}
The overset domain problem \eqref{eq:OmegaU-RSSystem-MPu}-\eqref{eq:OmegaV-RSSystem-MPv} is conservative if the terms in \eqref{eq:2DConservationEquation} along $\Gamma_b$ and $\Gamma_c$ vanish, which happens when
\begin{equation}
\beta \spacevec{\mmatrix A}\cdot\hat n + \mmatrix{$\Sigma$}_u = \mmatrix{$\Sigma$}_v
\label{eq:2DConsCondition}
\end{equation}
along each curve, where $\beta$ is given by \eqref{eq:BetaCondition2D}.
\label{thm:2DConservation}
\end{thm}
\begin{rem}
So, as in one space dimension, conservation is necessary for energy boundedness and energy boundedness is sufficient for conservation.
\end{rem}

\noindent Finally, as in one space dimension, 
\begin{thm}
Conservation in the overset domain problem \eqref{eq:OmegaU-RSSystem-MPu}-\eqref{eq:OmegaV-RSSystem-MPv} is equivalent to conservation in the original domain problem, \eqref{eq:Omega2DProblem}.
\label{thm:Conservation2DOriginal}
\end{thm}
\begin{proof}
Enforcing \eqref{eq:2DConsCondition} and choosing $\statevecGreek\phi = 1$, \eqref{eq:2DConservationEquation} becomes
\begin{equation}
\begin{split}
&\frac{d}{dt}\left\{\int_{\Omega_u} \statevec u dxdy+ \int_{\Omega_v} \statevec v dxdy-\eta\int_{\Omega_O} \statevec u dxdy -(1-\eta)\int_{\Omega_O} \statevec v dxdy\right\}
\\&= -\left\{ \oint_{\Gamma_a}\spacevec{\mmatrix A}\cdot\hat n_a \statevec u ds
+  \oint_{\Gamma_d}\spacevec{\mmatrix A}\cdot\hat n_d  \statevec v ds\right\}
\\&
= -\left\{ \oint_{\Gamma_a}\spacevec{\statevec f}\cdot\hat n_a ds
+  \oint_{\Gamma_d}\spacevec{\statevec f}\cdot\hat n_d ds\right\}.
\end{split}
\label{eq:2DConservation1}
\end{equation} 
Given the equivalence between the overset solutions and the full domain solution through Thm. \ref{thm:2DEquivalence}, i.e. $\statevec u = \statevecGreek\omega$ and  $\statevec v = \statevecGreek\omega$, \eqref{eq:2DConservation1} is equivalent to
\begin{equation}
    \frac{d}{dt}\int_\Omega \statevecGreek\omega dxdy=  -\left\{ \oint_{\Gamma_a}\spacevec{\statevec f}\cdot\hat n_a ds
+  \oint_{\Gamma_d}\spacevec{\statevec f}\cdot\hat n_d ds\right\},
\end{equation}
i.e., the rate of change of the total amount of $\statevecGreek\omega$ is given by the difference between the flux in and the flux out of the full domain, $\Omega$.

\end{proof}

\subsection{Boundary Plus Overlap Coupling}

To enable stronger coupling between the problems in the overlap region, we can add additional penalties in its interior. Let $\left\{\spacevec x_m\right\}^M_{m=1}$ be a set of points in the overlap region, $\Omega_O$. Then we modify \eqref{eq:OmegaU-RSSystem-MPu} and \eqref{eq:OmegaV-RSSystem-MPv} by adding penalties at those points as
\begin{equation}
O\;\left\{\begin{gathered}
\statevec u_t + \spacevec{\mmatrix A}\cdot\nabla\statevec u  + \mathcal L_u^b \left[ \mmatrix{$\Sigma$}_u^b (\statevec u-\statevec v)\right] 
+ \mathcal L_u^c \left[ \mmatrix{$\Sigma$}_u^c (\statevec u-\statevec v)\right] 
+ \frac{1}{M}\sum_{m=1}^M\mathcal L_u^m \left[ \mmatrix{$\Sigma$}_u^m (\statevec u-\statevec v)\right]
= 0, \quad\spacevec x\in\Omega_u\hfill\\
\statevec u(\spacevec x,t) = B^a\left( \statevec u(\spacevec x,t)\right),\quad \spacevec x \in \Gamma_a\hfill\\
\statevec u(\spacevec x,0) = \statevec u_0(x),\quad \spacevec x\in\Omega_u \hfill
\end{gathered} \right.
\label{eq:InteriorPenaltyStrongU}
\end{equation}
\begin{equation}
B\;\left\{\begin{gathered}
\statevec v_t + \spacevec{\mmatrix A}\cdot\nabla\statevec v + \mathcal L_v^b \left[ \mmatrix{$\Sigma$}_v^b (\statevec v-\statevec u)\right]
+ \mathcal L_v^c \left[ \mmatrix{$\Sigma$}_v^c (\statevec v-\statevec u)\right]
+ \frac{1}{M}\sum_{m=1}^M\mathcal L_v^m \left[ \mmatrix{$\Sigma$}_v^m (\statevec v-\statevec u)\right]
= 0, \quad\spacevec x\in\Omega_v\hfill\\
\statevec v(\spacevec x,t) = B^d\left( \statevec v(\spacevec x,t)\right)\quad \spacevec x \in \Gamma_d\hfill\\
\statevec v(x,0) = \statevec v_0(x),\quad \spacevec x\in\Omega_v ,\hfill
\end{gathered} \right.
\label{eq:InteriorPenaltyStrongV}
\end{equation}
where the new lifting operators, $\mathcal L^m_{*}$, select the values at the points $\spacevec x^m$,
\begin{equation}
    \int_{\Omega_O} \statevecGreek\psi^T\mathcal L_{*}^m(\phi)dxdy = \left.\statevecGreek\psi^T\statevecGreek\phi\right|_{\spacevec x^m}.
\end{equation}

The weak forms  \eqref{eq:OmegaU-RSSystem-MPuW}--\eqref{eq:OmegaV-RSSystem-MPvW} and \eqref{eq:2DWeakFormOveralap}
then become 
\begin{equation}
\begin{gathered}
\iprod{\statevecGreek\phi_u,\statevec u_t}_{\Omega_u}  + Q_u(\statevecGreek \phi_u,\statevec u, \statevec v) + \frac{1}{M}\sum_{m=1}^M\statevecGreek\phi_u^T\mmatrix{$\Sigma$}_u^m (\statevec u-\statevec v)= 0\hfill\\
\iprod{\statevecGreek\phi_v,\statevec v_t}_{\Omega_v}   + Q_v(\statevecGreek \phi_v,\statevec u, \statevec v)+ \frac{1}{M}\sum_{m=1}^M\statevecGreek\phi_v^T\mmatrix{$\Sigma$}_v^m (\statevec v-\statevec u) = 0\hfill\\
\iprod{\statevecGreek\phi_u,\statevec u_t}_{\Omega_O} + O_u(\statevecGreek \phi_u,\statevec u, \statevec v) + \frac{1}{M}\sum_{m=1}^M\statevecGreek\phi_u^T\mmatrix{$\Sigma$}_u^m (\statevec u-\statevec v)  = 0\hfill\\
\iprod{\statevecGreek\phi_v,\statevec v_t}_{\Omega_O}  + O_v(\statevecGreek \phi_v,\statevec u, \statevec v) +\frac{1}{M} \sum_{m=1}^M\statevecGreek\phi^T_v\mmatrix{$\Sigma$}_v^m (\statevec v-\statevec u)= 0,
\end{gathered}
\end{equation}
where
\begin{equation}
\begin{gathered}
Q_u(\statevecGreek \phi_u,\statevec u, \statevec v) \equiv  \iprod{\statevecGreek\phi_u, \spacevec{\mmatrix A}\cdot\nabla\statevec u}_{\Omega_u}   
+  \oint_{\Gamma_b} \statevecGreek\phi^T_u\mmatrix{$\Sigma$}_u^b (\statevec u-\statevec v)ds 
+ \oint_{\Gamma_c}\statevecGreek\phi^T_u\mmatrix{$\Sigma$}_u^c (\statevec u-\statevec v)ds\hfill\\
Q_v(\statevecGreek \phi_v,\statevec u, \statevec v) \equiv  \iprod{\statevecGreek\phi_v, \spacevec{\mmatrix A}\cdot\nabla\statevec v }_{\Omega_v}
+ \oint_{\Gamma_b} \statevecGreek\phi^T_v\mmatrix{$\Sigma$}_v^b (\statevec v-\statevec u)ds
+ \oint_{\Gamma_c} \statevecGreek\phi^T_v\ \mmatrix{$\Sigma$}_v^c (\statevec v-\statevec u)ds\hfill\\
O_u(\statevecGreek \phi_u,\statevec u, \statevec v) \equiv \iprod{\statevecGreek\phi_u, \spacevec{\mmatrix A}\cdot\nabla\statevec u}_{\Omega_O}\hfill\\
O_v(\statevecGreek \phi_v,\statevec u, \statevec v) \equiv \iprod{\statevecGreek\phi_v, \spacevec{\mmatrix A}\cdot\nabla\statevec v }_{\Omega_O}\hfill
\end{gathered}
\end{equation}
contain the parts already analyzed.
The previous sections on energy boundedness and conservation show that
\begin{equation}
Q_u(\statevec u,\statevec u, \statevec v) + Q_v(\statevec v,\statevec u, \statevec v) -\eta O_u(\statevec u,\statevec u, \statevec v) - (1-\eta)O_v(\statevec v,\statevec u, \statevec v) \ge 0
\end{equation}
and that
\begin{equation}
Q_u(1,\statevec u, \statevec v) + Q_v(1,\statevec u, \statevec v) -\eta O_u(1,\statevec u, \statevec v) - (1-\eta)O_v(1,\statevec u, \statevec v)= -\left\{ \oint_{\Gamma_a}\spacevec{\statevec f}\cdot\hat n_a ds
+  \oint_{\Gamma_d}\spacevec{\statevec f}\cdot\hat n_d ds\right\}
\end{equation}
when the trivial inflow boundary conditions are imposed and the coupling conditions \eqref{eq:1DSigmaMatrixConditionsW} hold. Therefore, to show energy boundedness and conservation with the penalty terms in the overlap region, we need only consider the consequences of those.

With the overlap penalty terms, the time derivative of the energy is
\begin{equation}
\begin{split}
 \frac{d}{dt}&\left\{\inorm{\statevec u}^2_{\Omega_u} + \inorm{\statevec v}^2_{\Omega_v} -\eta \inorm{\statevec u}^2_{\Omega_O}- (1-\eta)\inorm{\statevec v}^2_{\Omega_O}\right\} 
\\&
+ \frac{1}{M}\sum_{m=1}^M\left\{2(1-\eta) \statevec u^T\mmatrix{$\Sigma$}_u^m (\statevec u-\statevec v) + 2\eta\statevec u^T\mmatrix{$\Sigma$}_v^m (\statevec v-\statevec u)\right\}\le 0.
\end{split}
\label{eq:MultiPointPenaltyEnergyEqn}
\end{equation}
Therefore the overset problem is energy bounded if
\begin{equation}
\mathcal P_m = 2(1-\eta) \statevec u^T\mmatrix{$\Sigma$}_u^m (\statevec u-\statevec v) + 2\eta\statevec v^T\mmatrix{$\Sigma$}_v^m (\statevec v-\statevec u)\ge 0.
\end{equation}
To find the conditions on the coefficient matrices, we write $\mathcal P_m$ as
\begin{equation}
\mathcal P_m =\left[\begin{array}{c}\statevec u \\\statevec v\end{array}\right]^T\mmatrix M_m  \left[\begin{array}{c}\statevec u \\\statevec v\end{array}\right] ,
\end{equation}
where $\mmatrix M_m$ is independent of the advection coefficient matrices,
\begin{equation}
\mmatrix M_m =\left[\begin{array}{cc} 2(1-\eta)\mmatrix{$\Sigma$}^m_u& -((1-\eta)\mmatrix{$\Sigma$}^m_u  + \eta\mmatrix{$\Sigma$}^m_v)\\-((1-\eta)\mmatrix{$\Sigma$}^m_u+ \eta\mmatrix{$\Sigma$}^m_v) & -2\eta\mmatrix{$\Sigma$}^m_v \end{array}\right].
\end{equation}
As in Sec. \ref{sec:1DEnergyBoundedness}, the precise condition is 
found by rotating the matrix, where, now, 
\begin{equation}
\tilde{\mmatrix M}_m =\frac{1}{2} \left[\begin{array}{cc} 2((1-\eta)\mmatrix{$\Sigma$}^m_u+\eta\mmatrix{$\Sigma$}_v^m)& (1-\eta)\mmatrix{$\Sigma$}^m_u  - \eta\mmatrix{$\Sigma$}^m_v
\\(1-\eta)\mmatrix{$\Sigma$}^m_u- \eta\mmatrix{$\Sigma$}^m_v & 0 \end{array}\right].
\end{equation}
Following the arguments in Sec. \ref{sec:1DEnergyBoundedness}, the coupling matrices must be related by 
\begin{equation}
(1-\eta)\mmatrix{$\Sigma$}^m_u = \eta\mmatrix{$\Sigma$}_v^m,
\label{eq:InteriorPenaltyCoupling}
\end{equation}
and when satisfied, 
\begin{equation}
\mathcal P_m = 2(1-\eta)(\statevec u - \statevec v)^T\mmatrix{$\Sigma$}^m_u (\statevec u - \statevec v)\ge 0
\end{equation}
adds dissipation to the system. The amount of that dissipation can be controlled by the size of $\mmatrix{$\Sigma$}^m_u$ and $\mmatrix{$\Sigma$}^m_v$. 

We therefore have
\begin{thm}
For $\eta\in (0,1)$, and conditions \eqref{eq:1DSigmaMatrixConditionsW2}--\eqref{eq:1DSigmaMatrixConditionsW3}
on the interface coupling matrices, where $\beta$ is given by \eqref{eq:BetaCondition2D}, and \eqref{eq:InteriorPenaltyCoupling} on the interior penalty coupling matrices, the overset domain problem \eqref{eq:InteriorPenaltyStrongU}-\eqref{eq:InteriorPenaltyStrongV} is energy bounded with
$E(T) \le \inorm{\statevecGreek \omega_0}_\Omega^2$, and further, $\inorm{\statevec u}_{\Omega_u}^2 \le \inorm{\statevecGreek\omega_0}_\Omega^2/(1-\eta)$ and $\inorm{\statevec v}_{\Omega_v}^2 \le \inorm{\statevecGreek\omega_0}_\Omega^2/\eta$.
\end{thm}
\noindent We also state without proof, as the proof follows that of Thm. \ref{thm:1DEquivalenceWithPenalty},
\begin{thm}
For $\eta\in (0,1)$, and conditions \eqref{eq:1DSigmaMatrixConditionsW2}--\eqref{eq:1DSigmaMatrixConditionsW3}
on the interface coupling matrices, and \eqref{eq:InteriorPenaltyCoupling} on the interior penalty coupling matrices, the solutions of the overset domain problem \eqref{eq:InteriorPenaltyStrongU}-\eqref{eq:InteriorPenaltyStrongV} are equivalent to the solution for the full domain problem, i.e. $\statevec u = \statevecGreek\omega$ for $x\in\Omega_u$ and $\statevec v = \statevecGreek\omega$ for $x\in\Omega_v$. Under the assumption that the original problem is well-posed, the overset domain problem is well-posed.
\label{thm:Wellposed2DWithOverlapPenalty}
\end{thm}

\noindent Finally,
\begin{thm}
With the coupling conditions \eqref{eq:InteriorPenaltyCoupling}, the overset domain problem \eqref{eq:InteriorPenaltyStrongU}-\eqref{eq:InteriorPenaltyStrongV} remains conservative.
\label{thm:2DOverlapPenaltyConservation}
\end{thm} 
\begin{proof}
Under the condition \eqref{eq:InteriorPenaltyCoupling}, the interior penalty terms add
\begin{equation}
\statevecGreek\phi^T\left( (1-\eta)\mmatrix{$\Sigma$}^m_u - \eta\mmatrix{$\Sigma$}^m_v\right)(\statevec u - \statevec v) = 0,\quad m = 1, 2,\ldots, M
\end{equation}
at each point to \eqref{eq:2DConservationEquation} for any $\statevecGreek\phi$, $\statevec u$ and $\statevec v$, so Thm. \ref{thm:2DConservation} still holds.
\end{proof}

\begin{rem}
The addition of a finite number of penalty points is the strategy that one might encounter in numerical approximations. From the continuous perspective, however, one might consider replacing the sums with a continuous penalty that is non-zero only in the overlap region. With that penalty, the algebra carries through with integrals replacing the sums in \eqref{eq:InteriorPenaltyStrongU} and \eqref{eq:InteriorPenaltyStrongV}. The conditions on the coupling matrices remain the same, and the problem remains energy bounded and conservative.
\end{rem}

\section{Summary}

Overset grid methods, which have been used for four decades, have stability issues when applied to multidimensional problems. In this paper we stepped back from the numerical issue to the more fundamental one of whether a linear, constant coefficient hyperbolic initial-boundary value problem on which overset grid numerical methods are based is well-posed. Starting from a problem known to be well-posed is necessary for the development of a stable and accurate multidimensional overset grid approximation.

Using the energy method, we showed that characteristic coupling conditions between subdomains lead to a well-posed problem in one space dimension. For systems, the proof used the fact that the coefficient matrix can be diagonalized. If diagonalization is not used, then the energy method introduces interface terms that are not bounded by initial data.

To ensure well-posedness in multidimensional problems where requiring that the coefficient matrices be simultaneously diagonalizable is not usually possible, we introduced penalty terms at the domain interfaces (T3). Under specific conditions on the coupling matrices in those penalty terms, we showed that the overset domain problems are energy bounded, conservative and have solutions that are equivalent to the solution on the full domain. We use the latter to infer that that the overset domain problems are well-posed. The two tricks are to: (i) Define the energy over the entire domain by adding the energies of the subdomains, but subtracting the extra introduced by the overlap (T2) and (ii) use the fact that integration by parts is applicable in the overlap region to relate the interior and surface contributions (T1).

We also showed that stronger coupling introduced by adding penalty terms to the interior of the overlap region between their subdomains leads to a well-posed problem that is still equivalent to the original under specific conditions on the coupling matrices. This approach allows additional weighting of one subdomain solution over another, which may be useful in approximations where one solution is assumed to be more accurate than the other.

\section*{Acknowledgments}

 This work was supported by a grant from the Simons Foundation (\#426393, David Kopriva).  Jan Nordstr\"om was supported by Vetenskapsrådet, Sweden grant nr: 2018-05084 VR and the Swedish e-Science Research Center (SeRC). Gregor Gassner thanks the Klaus-Tschira Stiftung and
the European Research Council for funding through the ERC Starting Grant “An
Exascale aware and Un-crashable Space-Time-Adaptive Discontinuous Spectral
Element Solver for Non-Linear Conservation Laws” (EXTREME, project no. 71448).

\bibliographystyle{plain}
\bibliography{Chimera.bib}

\end{document}